\documentclass[a4paper,reqno]{amsart} 
\usepackage{amsmath, amsthm, amssymb, nicefrac, xfrac}
\usepackage[T1]{fontenc}
\usepackage[utf8]{inputenc}
\usepackage{mathrsfs}
\usepackage[margin=2.5cm]{geometry}
\usepackage[english]{babel}
\usepackage{mathtools}
\usepackage{scalerel}
\usepackage{stackengine,wasysym}
\usepackage{setspace}
\usepackage{esint}
\numberwithin{equation}{section}
\usepackage[backend=biber, style=alphabetic]{biblatex}
\newcommand\reallybigtilde[1]{\ThisStyle{%
		\setbox0=\hbox{$\SavedStyle#1$}%
		\stackengine{-.1\LMpt}{$\SavedStyle#1$}{%
			\stretchto{\scaleto{\SavedStyle\mkern.2mu\AC}{.41\wd0}}{.48\ht0}%
		}{O}{c}{F}{T}{S}%
}}
\newcommand{\Ae}{a\text{.}e\text{.}\ }
\theoremstyle{plain}
\newtheorem{thm}{Theorem}[section]
\newtheorem{prop}[thm]{Proposition}
\newtheorem{lemma}[thm]{Lemma}

\newtheorem{claim}[thm]{Claim}
\theoremstyle{definition}
\newtheorem{deff}[thm]{Definition}

\theoremstyle{remark}
\newtheorem{rmk}[thm]{Remark}
\newtheorem{xmp}[thm]{Example}

\theoremstyle{plain}
\newtheorem{thm*}{Theorem}
\newtheorem{lemma*}{Lemma}
\DeclareMathOperator{\Div}{\mathrm{div}}
\DeclareMathOperator{\N}{\mathbb{N}}

\DeclareMathOperator{\R}{\mathbb{R}}
\DeclareMathOperator{\B}{\mathbb{B}}

\DeclareMathOperator{\A}{\mathcal{A}}
\DeclareMathOperator{\J}{\mathcal{J}}
\DeclareMathOperator{\Sph}{\mathbb{S}}
\DeclareMathOperator{\ee}{\mathrm{e}}

\newcommand{\de}{\mathop{}\!\mathrm{d}}
\newcommand{\deH}{\mathop{}\!\mathrm{d}\mathcal{H}^{n-1}}

\newcommand{\dex}{\mathop{}\!\mathrm{d}x}
\newcommand{\Def}{\vcentcolon=}

\newcommand{\supp}{\mathrm{supp}}
\usepackage[a-1b]{pdfx}
\title[Local boundedness of weak solutions under unbalanced Orlicz growths]{Local boundedness of weak solutions to elliptic equations under unbalanced Orlicz growths}
\author{Gabriele Giannone}
\address{Dipartimento di Matematica e Informatica ``U. Dini'', Università di Firenze, Viale Morgagni 67/A, 50134 Firenze, Italy, e-mail: gabriele.giannone@unifi.it}

\begin{document}
\begin{abstract}
We establish sufficient conditions for the local boundedness of weak solutions to a broad class of nonlinear elliptic equations in divergence form, under unbalanced growth conditions on the stress field. Our analysis is carried out in a non-variational setting, with no symmetry or structural assumptions on the operator. The ellipticity and growth are prescribed via distinct Young functions, leading to a Orlicz-type setting that captures a wide class of nonstandard behaviors. As a special case, the theory encompasses and extends the known results on equations with $p,q$-growth.
\end{abstract}
\onehalfspacing
\maketitle
\section{Introduction}\label{introduction}
In this paper, we provide a complete proof of the local boundedness of weak solutions to second order elliptic equations in divergence form of the type
\begin{equation}\label{equation}
-\Div(\A(x,\nabla u))=f(x,u)\quad \text{in}\ \Omega,
\end{equation}
where $\Omega$ is an open set in $\R^n$, with $n\ge2$, and $\A\colon \Omega\times\R^n\to\R^n$ and $f\colon\Omega\times\R\to\R$ are Carathéodory maps that satisfy proper growth conditions.

Local boundedness of weak solutions to the equation in \eqref{equation} is guaranteed if $\A$ is subject to lower and upper bounds --- named $p$-growth or natural conditions --- of the form
\begin{gather}\label{p growth}
\A(x,\xi)\cdot\xi \gtrsim |\xi|^{p}\quad\text{and}\quad|\A(x,\xi)|\lesssim |\xi|^{p-1}\quad \text{for \Ae} x\in\Omega\ \text{and all } \xi\in\R^n,
\end{gather}
for some exponent $p>1$. A result of this kind can be found in \cite{Trudinger1967OnHT}, where, in the spirit of \cite{Moser1961OnHT}, an approach based on Harnack type inequalities is adopted. The local Lipschitz continuity of weak solutions is proved in Chapter $4$ of the classical reference book \cite{LadUral68}, employing techniques that, in turn, rely upon methods introduced by De Giorgi \cite{DeGiorgi57} in his regularity theory for linear elliptic equations with merely measurable coefficients.

A regularity theory for elliptic equations whose stress field exhibits different growth rates from below and above --- commonly known as $p,q$-growth --- originates in the work \cite{Kolodii}, which deals with the local boundedness of solutions to certain anisotropic equations. In the literature on this topic, a typical assumption is that \eqref{equation} arises as the Euler–Lagrange equation of a variational integral of the form
$$\mathscr{F}(u,\Omega)=\int\limits_\Omega \mathcal{L}(x,u,\nabla u)\dex,$$
with
$$|\xi|^p-1\lesssim \mathcal{L}(x,z,\xi)\lesssim |\xi|^q+1$$
for \Ae $x\in\R$, every $z\in\R$, and every $\xi\in\R^n$. This corresponds to requiring that $\A_i(x,\xi)=\partial_{\xi_i}\mathcal{L}(x,z,\xi)$, for $i=1,\dots,n$, and $f(x,z)=\partial_{z}\mathcal{L}(x,z,\xi)$. The case $p=q>1$ is considered in the paper \cite{GiaquintaGiusti84}, where, assuming the growth condition only, every quasi-minimizer is proved to be locally H\"older continuous. The study of regularity properties for local minimizers under unbalanced growth conditions began with the works \cite{Marcellini1989,Marcellini1991,Marcellini1993}, and has since received considerable attention. As it is well known, most results in the existing literature require some restrictions on the ratio $q/p$; typically, one assumes
$$\frac{q}{p}<c_n, \ \text{for some constant}\ c_n=c_n(p,q)\ \text{such that } c_n\downarrow1\ \text{as}\ n\to\infty.$$
Under various structure conditions on the lagrangian $\mathcal{L}$, the local boundedness of minimizers of such functionals is treated in \cite{MoscarielloNania,MascoloPapi,CupiniMarcelliniMascolo2015}, while similar questions are addressed in \cite{BoccardoMarcelliniSbordone,FuscoSbordone93,Stroffolini} for problems with anisotropic growths. On the contrary, counterexamples in \cite{Giaquinta1987,Marcellini1987,Marcellini1991} demonstrate the existence of unbounded minimizers when $p$ and $q$ are too far apart. More specifically, Theorem $6.1$ in \cite{Marcellini1991} constructs unbounded solutions to the Euler-Lagrange equation of a variational integral with $p,q$-growth, under the assumption
$$q>\frac{(n-1)p}{n-1-p}=\vcentcolon p^{\star}_{n-1}.$$
The gap between the assumptions on $p$ and $q$ in these examples and in the regularity results available until recently has been filled in the paper \cite{HirschSchaffner}, where the local boundedness of minimizers is established for the full range $q\le p^{\star}_{n-1}$.

Less common in the literature is the study of differential equations that are not the Euler first variation of an energy integral. Examples of such problems can be found, for instance, in \cite{CupiniMarcelliniMascolo24ARMA}. A distinctive feature of the non-variational setting, unlike the variational one, is that although Equation \eqref{equation} remains elliptic and coercive in $W^{1,p}_\mathrm{loc}(\Omega)$ under $p,q$-growth, the notion of weak solution is consistent only for functions that belong, a priori, to $W^{1,q}_\mathrm{loc}(\Omega)$. The issues of local Lipschitz continuity and $C^\infty$ regularity of weak solutions are addressed in \cite{Marcellini1991}, whereas the boundedness of weak solutions to Dirichlet problems under anisotropic growth conditions is studied in \cite{BoccardoMarcelliniSbordone}. More recently, adapting the strategy of \cite{HirschSchaffner}, the authors of \cite{CupiniMarcelliniMascolo23} proved local boundedness of weak solutions under the balance condition
$$q<\frac{np}{n-1}.$$

In the present paper, without assuming any variational structure for the equation, we address the question of local boundedness of weak solutions under more general assumptions on the stress field: namely, unbalanced Orlicz-type growth conditions are imposed on $\A$. This means that the growth and ellipticity of $\A$ are dictated by Young functions --- i.e., non-negative convex functions vanishing at the origin. When the growth rate of the stress field is prescribed in terms of the same Young function, local boundedness results have been obtained in \cite{Korolev1990}, and local H\"older continuity of bounded solutions is proved in \cite{Lieberman1991}. The boundedness of local minimizers for variational integrals with anisotropic Orlicz growths is treated in \cite{Cianchi2000147}.

In contrast, we prescribe the ellipticity and growth of $\A$ via two \emph{different} Young functions $A$ and $B$, with the model assumption:
\begin{gather*}
\A(x,\xi)\cdot\xi \ge A(|\xi|)\quad\text{and}\quad|\A(x,\xi)|\le b(|\xi|)\quad \text{for \Ae} x\in\Omega\ \text{and all } \xi\in\R^n,
\end{gather*}
where $b(|\xi|)=B'(|\xi|)$ for \Ae $\xi\in\R^n$. Natural conditions are also imposed on the reaction term $f(x,u)$. As in the polynomial-growth setting, the relevant notion of weak solution is well-posed only in the Orlicz-Sobolev class $W^1L^B_\mathrm{loc}(\Omega)$. The same problem for local quasi-minimizers of variational integrals under these assumptions has been studied in \cite{CianchiSchaffner}, where a sharp balance condition between the Young functions $A$ and $B$ is exhibited. The results for functionals with $p,q$-growth emerge as a special case of this framework.

The interest in establishing local boundedness of weak solutions stems from the higher regularity theory, where local boundedness serves as a starting assumption to prove further interior regularity properties of $u$ --- including higher integrability, local Lipschitz continuity, local $C^{1,\alpha}$ regularity, and beyond. In this regard, a strong motivation for our study lies in the paper \cite{Marcellini1993}, that, under pointwise unbalanced bounds on the matrix $(\partial_{\xi_j}\A_i(x,\xi))$ of possibly non-polynomial rate, establishes the local Lipschitz continuity of any weak solution that is \emph{assumed} to belong to $L^\infty_\mathrm{loc}(\Omega)$. The crucial role of local boundedness results in higher regularity theory is discussed, for instance, in \cite{BaroniColomboMingione,BousquetBrasco,DeFilippisMingione21AMP}.

As mentioned above, beyond boundedness, many aspects of the regularity theory for nonlinear elliptic equations and systems with unbalanced growth conditions have been investigated. Without aiming for completeness, we refer to \cite{EspositoLeonettiMingione1999,EspositoLeonettiMingione2004,ColomboMingione2015,BaroniColomboMingione,BousquetBrasco,BeckMingione,BellaSchaffner2020,DeFilippisMingione21AMP,,DeFilippisMingione21ARMA,DeFilippisMingione2023,BellaSchaffner2024}, and point the reader to the surveys \cite{MingioneRadulescu,Marcellini2021} which outline recent developments and open problems in the field.

\subsection{Organization of the paper} The paper is organized as follows: Section \ref{Orlicz-Sobolev spaces} collects the basic notations and results concerning Orlicz spaces; in Section \ref{main results}, we delineate the precise framework of the problem and present the main theorem, accompanied by relevant comments; Section \ref{Technical lemmas} contains the technical preliminaries necessary for the proof, which is carried out in Section \ref{Proof of the Theorem}.

\section{Background}\label{Orlicz-Sobolev spaces}
In this section, we recall some basic definition and property from the theory of Young functions, Orlicz spaces, and Orlicz-Sobolev spaces, with special attention to those results that play a crucial role in our proofs. We refer the reader to \cite{RaoRen,RaoRenApplications} for a comprehensive presentation of this theory.
\subsection{Young functions} A function $A\colon[0,\infty)\to[0,\infty]$ is said to be a Young function if it is convex, left-continuous, neither identically zero nor infinity, and $A(0)=0$. This very definition implies that
$$t\mapsto\frac{A(t)}{t}$$
is increasing in $[0,\infty)$. Note that $A$ is allowed to jump to infinity, which is the only case in which the left-continuity assumption is relevant.

We say that a Young function has \emph{super-linear} growth if
$$\lim_{t\to\infty}\frac{A(t)}{t}=\infty,$$
otherwise we say that $A$ has \emph{linear} growth.

Every Young function $A$ admits the representation
\begin{equation}\label{integral representation}
A(t)=\int\limits_0^t a(s)\de s\quad\text{for all}\ t\in[0,\infty)
\end{equation}
where $a\colon[0,\infty)\to[0,\infty]$ is an increasing, left-continuous function such that $a(0)=0$. It holds that
\begin{equation}\label{inequalities for A and a}
A(t)\le ta(t)\le A(2t).
\end{equation}
We call the function $a$ the \emph{left-derivative} of $A$. This terminology comes from the trivial observation that, in the interval where $A$ is finite,
$$A(t)=\int\limits_0^t A'_{-}(s)\de s$$
where $A'_{-}(\cdot)$ denotes the left-derivative of $A$, which exists finite at $s$ as soon as $A(s)$ is finite.\\
Throughout this paper, Young functions will be denoted by capital letters, and their left-derivatives by the corresponding lower case letters.

The Young conjugate of \emph{any} function $\Phi\colon[0,\infty)\to[0,\infty]$ is the \emph{convex} function given by
$$\widetilde{\Phi}(s)=\sup\{ts-\Phi(t) : t\ge0\}\quad\text{for all}\ t\in[0,\infty).$$
An immediate consequence of this definition is the so-called Young’s inequality
$$ts\le \Phi(t)+\widetilde{\Phi}(s)\quad\text{for all}\ t,s\ge0.$$
If $A$ is a Young function, one can verify that
$$\widetilde{A}(t)=\int\limits_0^t a^{-1}(s)\de s\quad\text{for}\ t\in[0,\infty)$$
where $a^{-1}$ is the generalized left-continuous inverse of $a$, namely
\begin{equation*}
	a^{-1}(s)=\inf\{t\ge0 : a(t)\ge s\}.
\end{equation*}
Moreover, it holds that
\begin{equation}\label{inequality between A conjugate and A}
\widetilde{A}(a(t))\le ta(t)
\end{equation}
and
\begin{equation}\label{inequality between inverses of A conjugate and A}
t\le A^{-1}(t)\widetilde{A}^{-1}(t)\le 2t
\end{equation}
for all $t\ge0$, where $A^{-1}$ and $\widetilde{A}^{-1}$ denote the generalized right-continuous inverses of $A$ and $\widetilde{A}$. Recall that, if $\Phi\colon[0,\infty)\to[0,\infty]$ is increasing, such an inverse is given by
\begin{equation}\label{generalized right continuous inverse}
\Phi^{-1}(s)\Def\inf\{t\ge0 : \Phi(t)>s\}\quad \text{for all}\ s\ge0.
\end{equation}

A Young function $A\colon[0,\infty)\to[0,\infty]$ is said to satisfy the $\Delta_2$-condition globally --- briefly $A\in\Delta_2$ globally --- if there exists a constant $c$ such that
$$A(2t)\le cA(t)\quad \text{for all}\ t\ge0.$$
We say that $A$ satisfies the $\Delta_2$-condition near infinity if the latter inequality holds for large $t$ only. One can verify that
\begin{gather}\label{equivalent definition of Delta2}
A\in\Delta_2\ \text{near infinity iff there exist}\ t_0\ge0\ \text{and}\ q\ge1\ 
\text{such that}\ \sup_{t\ge t_0}\frac{ta(t)}{A(t)}\le q.
\end{gather}

We will take advantage of a notion of index of a Young function related to its asymptotic growth behavior. We define the lower Matuszewska--Orlicz index of $A$ at infinity as
\begin{equation}\label{upper Matuszewska--Orlicz at infinity}
i_\infty(A)=\lim_{\mu\to+\infty}\frac{\log\big(\liminf_{t\to+\infty}\frac{A(\mu t)}{A(t)}\big)}{\log\mu}.
\end{equation}
The upper Matuszewska--Orlicz index of $A$ at infinity is defined similarly to \eqref{upper Matuszewska--Orlicz at infinity}, but with $\liminf_{t\to+\infty}$ replaced by $\limsup_{t\to+\infty}$. A crucial property of these indices is that they are invariant under replacement of $A$ by any equivalent Young function near infinity. In general, one has that
$$1\le i_\infty(A)\le I_\infty(A)\le\infty.$$
Clearly, if $A(t)\approx t^p(\log t)^\alpha$ near infinity, where either $p>1$ and $\alpha\in\R$, or $p=1$ and $\alpha\ge0$, then $i_\infty(A)=I_\infty(A)=p$. An example of a Young function exhibiting oscillatory behavior at infinity is
$$A(t)\approx t^{p+\alpha\sin(\log\log t)}\quad\text{near infinity},$$
with $p>1+\alpha\sqrt{2}$. Variational integrals whose growth is described by such functions have been considered in \cite{Marcellini1993,MascoloPapi}.

A function $A$ dominates another function $B$ globally (respectively, near infinity) if there exists a positive constant $c$ such that $B(t)\le A(ct)$ for all $t\ge0$ (respectively, for all $t\ge t_0\ge0$). If this is the case, we write $A\lesssim B$ globally (respectively, near infinity). $A$ and $B$ are said to be equivalent globally (respectively, near infinity) if $A$ dominates $B$ and vice versa globally (respectively, near infinity), and we write $A\approx B$ globally (respectively, near infinity).

\subsection{Orlicz and Orlicz-Sobolev spaces}
Let $A$ be a Young function, and $\Omega$ be a measurable set in $\R^n$. The Orlicz class $K^A(\Omega)$ is defined by
$$K^A(\Omega)=\bigg\{u: u\ \text{is measurable in}\ \Omega\ \text{and}\ \int\limits_\Omega A(|u|)\dex<\infty\bigg\}.$$
The Orlicz space $L^A(\Omega)$ built upon $A$ is the linear span of $K^A(\Omega)$. $L^A(\Omega)$ is a Banach function space, endowed with the Luxemburg norm
$$\|u\|_{L^A}=\inf\bigg\{\tau>0 :\int\limits_\Omega A\left(\frac{|u|}{\tau}\right)\dex\le1\bigg\}.$$
$K^A(\Omega)$ is a linear space --- that is, $K^A(\Omega)=L^A(\Omega)$ --- if and only if either $A\in\Delta_2$ globally, or $A\in\Delta_2$ near infinity and $|\Omega|<\infty$.

If $A(t)=t^p$ for some $p\in[1,\infty)$, then $L^A(\Omega)=L^p(\Omega)$, whereas $L^A(\Omega)=L^\infty(\Omega)$ if $A(t)=0$ for $t\in[0,1]$ and $A(t)=\infty$ for $t\in(1,\infty)$.

The continuous embedding
$$L^B(\Omega)\longrightarrow L^A(\Omega)$$
holds if either $B$ dominates $A$ globally, or $B$ dominates $A$ near infinity and $|\Omega|<\infty$.

Assume now that $\Omega$ is an open set. The (inhomogeneous) Orlicz-Sobolev class $W^1K^A(\Omega)$ is defined as
\begin{equation}\label{orlicz-sobolev class}
W^1K^A(\Omega)=\{u\in K^A(\Omega) : |\nabla u|\in K^A(\Omega)\}.
\end{equation}
Accordingly, the (inhomogeneous) Orlicz-Sobolev space $W^1L^A(\Omega)$ is given by
\begin{equation}\label{orlicz-sobolev space}
W^1L^A(\Omega)=\{u\in L^A(\Omega) : |\nabla u|\in L^A(\Omega)\}.
\end{equation}
The latter is a Banach space equipped with the norm
$$\|u\|_{W^{1}L^A}=\|u\|_{L^A}+\|\nabla u\|_{L^A}.$$
In the case when $L^A(\Omega) =L^p(\Omega)$ for some $p\in[1,\infty]$, the standard Sobolev space $W^{1,p}(\Omega)$ is recovered.

The local versions $K^A_\mathrm{loc}(\Omega)$, $L^A_\mathrm{loc}(\Omega)$, $W^1K^A_\mathrm{loc}(\Omega)$, and $W^1L^A_\mathrm{loc}(\Omega)$ are defined in the usual way by localizing the definitions given above. Membership of a function in the local Orlicz class $K^A_\mathrm{loc}$ is unaffected if $A$ is replaced by any Young function equivalent to it near infinity. This is because the behavior of $A$ near zero plays no role when integrals are taken over sets of finite measure.

Our analysis makes use of Orlicz-Sobolev classes of weakly differentiable functions defined on the sphere $\Sph^{n-1}$. These classes, denoted by $W^1K^A(\Sph^{n-1})$, are defined as in \eqref{orlicz-sobolev class}, except that the Lebesgue measure is replaced by the $(n-1)$-dimensional Hausdorff measure $\mathcal{H}^{n-1}$, and $\nabla u$ is replaced by $\nabla_{\Sph}u$, the weak covariant derivative of $u$ on $\Sph^{n-1}$. For a detailed treatment of Sobolev spaces on manifolds, see, e.g., \cite{Hebey}.

Sharp embedding theorems in Orlicz-Sobolev spaces play a key role in the formulation of our main result. Moreover, the proof relies heavily on the corresponding sharp inequalities. As shown in \cite{Cianchi1997} --- and, in an equivalent formulation, in  \cite{Cianchi1996} --- the optimal $n$-dimensional Sobolev conjugate of a Young function $A$ satisfying
\begin{equation}\label{integral near zero for definition of conjugate}
\int\limits_0 \left(\frac{t}{A(t)}\right)^{\frac{1}{n-1}}\de t<\infty
\end{equation}
is given by
\begin{equation}\label{definition of d conjugate}
A_{n}(t)=A(H_{n}^{-1}(t))\quad\text{for}\ t\ge0,
\end{equation}
where the function $H_{n}\colon[0,\infty)\to[0,\infty)$ is defined by
\begin{equation}\label{definition of H_d}
H_{n}(s)=\Big(\int_0^s\left(\frac{t}{A(t)}\right)^{\frac{1}{n-1}}\de t\Big)^{\frac{1}{n'}}\quad\text{for}\ s\ge0.
\end{equation}
For $n\ge3$, the $(n-1)$-dimensional Sobolev conjugate is defined analogously, with $n$ replaced throughout by $n-1$ in \eqref{integral near zero for definition of conjugate}, \eqref{definition of d conjugate}, and \eqref{definition of H_d}.

If $\rho$ is a positive radius, we denote the mean of $u$ over the ball $\B_\rho=\B_\rho(0)$ by
$$(u)_{\B_\rho}=\fint\limits_{\B_\rho} u\dex=\frac{1}{\omega_n\rho^n}\int\limits_{\B_\rho} u\dex,$$
where $\omega_n$ is the measure of the unit ball.

Both Sobolev-Poincaré and Sobolev inequalities on balls in $\R^n$ will be used in our proof. They are stated in the next two results, and their proofs can be found in Theorems $3.1$ and $4.4$ in \cite{Cianchi2006745}.
\begin{thm}\label{Orlicz-Sobolev embedding}
Let $n\ge2$, let $\rho>0$, and let $A$ be a Young function fulfilling the condition \eqref{first case main thm}.

$(i)$ Assume that
$$\int\limits^\infty \Big(\frac{t}{A(t)}\Big)^{\frac{1}{n-1}}\de t=\infty.$$
Then, there exists a constant $\gamma=\gamma(n,\rho)>0$ such that
\begin{equation}\label{first case orlicz-sobolev}
\int\limits_{\B_\rho}A_n\Bigg(\frac{|u|}{\gamma \big(\int_{\B_\rho}A(|u|)+A(|\nabla u|)\dex\big)^{\frac{1}{n}}}\Bigg)\dex\le\int\limits_{\B_\rho}A(|u|)+A(|\nabla u|)\dex
\end{equation}
for every $u\in W^1K^A(\B_\rho)$.

$(ii)$ Assume that
$$\int\limits^\infty \Big(\frac{t}{A(t)}\Big)^{\frac{1}{n-1}}\de t<\infty.$$
Then, there exists a constant $\gamma=\gamma(n,\rho,A)>0$ such that
$$\|u\|_{L^\infty(\B_\rho)}\le\gamma\bigg(\int\limits_{\B_\rho}A(|u|)+A(|\nabla u|)\dex\bigg)^{\frac{1}{n}}$$
for every $u\in W^1K^A(\B_\rho)$.

In particular, if $\rho\in[r_1,r_2]$, the constant $\gamma$ depends on $\rho$ only via $r_1$ and $r_2$.
\end{thm}
\begin{thm}\label{Sobolev-Poincaré inequality}
Let $n\ge2$, let $\rho>0$, and let $A$ be a Young function fulfilling the condition \eqref{first case main thm}. Then there exists a constant $\gamma=\gamma(n)>0$ such that
$$\int\limits_{\B_\rho} A_n\Bigg(\frac{|u-(u)_{\B_\rho}|}{\gamma\big(\int_{\B_\rho}A(|\nabla u|)\dex\big)^{\frac{1}{n}}}\Bigg)\dex\le \int\limits_{\B_\rho}A(|\nabla u|)\dex$$
for every $u\in W^1K^A(\B_\rho)$.
\end{thm}
Foundational to our approach is the following analogue of Theorem \ref{Orlicz-Sobolev embedding} for Orlicz-Sobolev functions on $\Sph^{n-1}$. As in Theorem \ref{Orlicz-Sobolev embedding}, we present the statement in two parts. The first part corresponds to $(i)$ of Theorem $\mathrm{C}$ in \cite{CianchiSchaffner}, while minor modifications in the proof of $(ii)$ of that theorem lead to the second part.
\begin{thm}\label{Orlicz-Sobolev on spheres}
Let $n\ge2$ and let $A$ be a Young function such that
$$\int\limits_0\Big(\frac{t}{A(t)}\Big)^{\frac{1}{n-2}}\dex<\infty\quad\text{if}\ n\ge3.$$

$(i)$ Assume that
$$n\ge3\quad\text{and}\quad \int\limits^\infty\Big(\frac{t}{A(t)}\Big)^{\frac{1}{n-2}}\dex=\infty.$$
Then, there exists a constant $\gamma=\gamma(n)>0$ such that
$$\int\limits_{\mathbb{S}^{n-1}}A_{n-1}\Bigg(\frac{|u|}{\gamma \big(\int_{\mathbb{S}^{n-1}}A(|u|)+A(|\nabla_{\mathbb{S}} u|)\deH\big)^{\frac{1}{n-1}}}\Bigg)\dex\le\int\limits_{\mathbb{S}^{n-1}}A(|u|)+A(|\nabla_{\mathbb{S}} u|)\deH$$
for every $u\in W^1K^A(\Sph^{n-1})$.

$(ii)$ Assume that there exist $p\in[1,\infty)$ and a positive constant $c$ such that $A(t)\ge ct^p$ for $t\ge0$. Let one of the following situations occur:
$$\begin{cases}
n=2,\\
n\ge 3 \quad\text{and}\quad \displaystyle\int\limits^\infty\Big(\frac{t}{A(t)}\Big)^{\frac{1}{n-2}}\dex<\infty.
\end{cases}$$
Then, there exists a constant $\gamma=\gamma(n,p,A)>0$ such that
$$\|u\|_{L^\infty(\Sph^{n-1})}\le\gamma\bigg(\int\limits_{\mathbb{S}^{n-1}}A(|u|)+A(|\nabla_{\mathbb{S}} u|)\deH\bigg)^\frac{1}{p}$$
for every $u\in W^1K^A(\Sph^{n-1})$.
\end{thm}

\section{Main result}\label{main results}
Let $A$ and $B$ be Young functions. We impose an $A,B$-growth condition on $\A$, namely, we assume
\begin{equation}\label{ellipticity}
\A(x,\xi)\cdot\xi\ge A(|\xi|)\quad\text{for \Ae} x\in\Omega\ \text{and all}\ \xi\in\R^n,\ |\xi|\ge R,
\end{equation}
and
\begin{equation}\label{growth}
	|\A(x,\xi)|\le b(|\xi|)+\phi(x)\quad\text{for \Ae} x\in\Omega\ \text{and all}\ \xi\in\R^n,
\end{equation}
where $R\ge0$ is a constant, $\phi$ is a non-negative measurable function satisfying suitable integrability assumptions, and $b\colon[0,\infty)\to[0,\infty]$ is the left-derivative of $B$ (whose definition has been recalled in Section \ref{Orlicz-Sobolev spaces}).

Concerning the reaction term, we assume that there exist a Young function $E$ and a non-negative measurable function $\psi$, satisfying conditions to be specified later, such that
\begin{equation}\label{growth on f}
|f(x,z)|\le \psi(x)+\frac{E(|z|)}{|z|}\quad\text{for \Ae} x\in\Omega\ \text{and all}\ z\in\R.
\end{equation}

Throughout the paper, we assume that
\begin{equation}\label{integral at infinity is infinite}
\int\limits^\infty \left(\frac{t}{B(t)}\right)^{\frac{1}{n-1}}\de t=\infty,
\end{equation}
otherwise part $(ii)$ of Theorem \ref{Orlicz-Sobolev embedding} implies that every function $u\in W^1L^{B}_{\mathrm{loc}}(\Omega)$ is locally bounded, independently of whether it is a weak solution to Equation \eqref{equation} or not.

A useful consequence of \eqref{integral at infinity is infinite} is that the lower Matuszewska--Orlicz index of $B$ at infinity, $i_\infty(B)$, defined in \eqref{upper Matuszewska--Orlicz at infinity}, satisfies
\begin{equation}\label{lower index of B is smaller than n}
i_\infty(B)\le n.
\end{equation}
To see this, first note that $i_\infty(B)$ must be finite; otherwise, by the very definition of the index, $B$ would grow faster than any polynomial, contradicting \eqref{integral at infinity is infinite}. Now, given any $\varepsilon>0$, again by the definition of $i_\infty(B)$ it follows that
\begin{equation}\label{B dominates small powers}
B\ \text{dominates}\ t^{i_\infty(B)-\varepsilon}\ \text{near infinity.}
\end{equation}
If, by contradiction, $i_\infty(B)$ were strictly bigger than $n$, then \eqref{integral at infinity is infinite} and \eqref{B dominates small powers} would be incompatible for  $0<\varepsilon<i_\infty(B)-n$. This proves \eqref{lower index of B is smaller than n}.
\subsection{Setting of the problem}
From now on, $b^{-1}\colon[0,\infty)\to[0,\infty]$ will denote the generalized left-continuous inverse of the left-derivative $b$, i.e.,
$$b^{-1}(s)=\inf\{t\ge0 : b(t)\ge s\}\quad\text{for all}\ s\ge0.$$
We assume some restrictions on the growth rates of $B$ and $E$. Precisely, we assume that
\begin{equation}\label{Delta2 assumption}
B\ \text{satisfies the } \Delta_2\text{-condition near infinity}
\end{equation}
and
\begin{equation}\label{Delta2 on E}
E\ \text{is super-linear and satisfies the } \Delta_2\text{-condition near infinity}.
\end{equation}
Moreover, a natural growth condition at infinity is imposed on $E$, meaning that we suppose
\begin{equation}\label{domination on E}
E(t)\le A_n(ct)\quad\text{for all}\ t\ge t_0
\end{equation}
for some $c>0$ and $t_0\ge0$, where $A_n$ is the optimal $n$-dimensional Sobolev conjugate of $A$.

The $\Delta_2$-condition ensures that $B$ does not grow too fast at infinity. In particular, by \eqref{equivalent definition of Delta2}, there exist $q\ge1$, $c>0$, and $t_1\ge0$ such that
\begin{equation}\label{domination of derivative}
b(t)\le q\frac{B(t)}{t}\quad \text{for all}\ t\ge t_1
\end{equation}
and
\begin{equation}\label{estimate of B with power}
B(t)\le c t^{q}\quad\text{for all}\ t\ge t_1.
\end{equation}
As a consequence, the left-derivative $b$ is finite-valued, and $B$ grows at most polynomially. Moreover, $B$ has linear growth if and only if $b$ is bounded.

From \eqref{Delta2 assumption}, we also infer that the upper Matuszewska--Orlicz index $I_\infty(B)$ is finite. This follows from the inequality
$$\frac{B(\mu t)}{B(t)}\le c^{\log_2\mu}$$
which holds for all sufficiently large $t$, $\mu$ and for a suitable constant $c$.

By \eqref{integral at infinity is infinite} and \eqref{Delta2 assumption}, we deduce that
\begin{equation}\label{integral at infinity of A}
\int\limits^\infty\bigg(\frac{t}{A(t)}\bigg)^\frac{1}{n-1}\de t=\infty.
\end{equation}
Indeed, by \eqref{ellipticity}, \eqref{growth}, and \eqref{domination of derivative},
\begin{equation}\label{B dominates A}
A(t)\le q B(t)+\phi(x) t\quad\text{for \Ae}x\in\Omega\ \text{and all}\ t\ge R.
\end{equation}
Since $B$ is convex, there exists a constant $c=c(B)>0$ such that $t\le c B(t)$ for large $t$. Hence, \eqref{B dominates A} implies
\begin{equation}\label{B dominates A for large t}
A(t)\le (q+\phi(x))B(t)\quad\text{for \Ae} x\in\Omega\ \text{and all large}\ t.
\end{equation}
Combining \eqref{B dominates A for large t} with \eqref{integral at infinity is infinite}, we obtain for \Ae $x\in\Omega$
$$\int\limits^\infty\bigg(\frac{t}{A(t)}\bigg)^\frac{1}{n-1}\de t\ge \frac{1}{(q+\phi(x))^\frac{1}{n-1}} \int\limits^\infty\bigg(\frac{t}{B(t)}\bigg)^\frac{1}{n-1}\de t=\infty,$$
so that \eqref{integral at infinity of A} follows. In particular, repeating the argument used to prove \eqref{lower index of B is smaller than n}, we conclude that $i_\infty(A)$ does not exceed $n$.

\begin{rmk}\label{A has a polynomial growth}
Let us mention that \eqref{estimate of B with power} ensures $A$ has a polynomial growth too. This follows easily once one observes that, integrating \eqref{B dominates A for large t} over any ball $\B\Subset\Omega$,
$$A(t)\le (q+(\phi)_{\B}) B(t)\quad\text{for all large}\ t,$$
where $(\phi)_{\B}$ denotes the integral average of $\phi$ on $\B$. Nevertheless, the $\Delta_2$-condition is not required on $A$.
\end{rmk}

Further, replacing, if necessary, the functions $A$ and $B$ by equivalent Young functions near infinity (see Lemma \ref{redefinition of functions and parameters}), we may assume without loss of generality that:

$(i)$ $A(t),B(t)>0$ for all $t>0$; this, together with \eqref{ellipticity} and the $\Delta_2$-condition on $B$, yields
$$0<A(t),B(t)<\infty\quad\text{for all}\ t>0,$$
whence $A$ and $B$ are locally Lipschitz continuous and strictly increasing on $[0,\infty)$.

$(ii)$ The integral
\begin{equation*}
\int\limits_0 \left(\frac{t}{A(t)}\right)^{\frac{1}{n-1}}\de t
\end{equation*}
is finite. Then, the optimal $n$-dimensional Sobolev conjugate of $A$ is given by \eqref{definition of d conjugate}.

$(iii)$ There exists $\varepsilon\in[0,i_\infty(A)-1]$ such that
\begin{equation}\label{lower growth on A by power}
t^{i_\infty(A)-\varepsilon}\le cA(t)\quad\text{for all}\ t\ge0
\end{equation}
for some $c=c(A,n)>0$. Set $p=i_\infty(A)-\varepsilon$, inequality \eqref{lower growth on A by power} immediately implies
\begin{equation}\label{sobolev conjugate dominates the conjugate of the lower exponent}
A_n(t)\gtrsim t^{\frac{np}{n-p}}\quad\text{globally}.
\end{equation}
Indeed, given $t\ge0$, one has
\begin{align*}
H_n(t)=\bigg(\int\limits_0^{t} \bigg(\frac{s}{A(s)}\bigg)^{\frac{1}{n-1}}\de s\bigg)^{\frac{n-1}{n}}\le c^{\frac{1}{n}}\frac{n}{n-p}t^{\frac{n-p}{n}}.
\end{align*}
It follows that
$$A_n(t)=A(H_n^{-1}(t))\ge A\Big(c^{-\frac{1}{n-p}}\Big(\frac{n-p}{n}\Big)^{\frac{n}{n-p}}t^{\frac{n}{n-p}}\Big)\ge c^{-\frac{p}{n-p}}\Big(\frac{n-p}{n}\Big)^{\frac{np}{n-p}}t^{\frac{np}{n-p}}\quad\text{for all}\ t\ge0,$$
which proves \eqref{sobolev conjugate dominates the conjugate of the lower exponent}.

Finally, the measurable functions $\psi$ and $\phi$ are assumed to satisfy proper integrability conditions. To prescribe the latter, we take advantage of the Matuszewska--Orlicz indices of $A$ and $B$ at infinity. Our assumptions are the following:
\begin{equation}\label{requirements on psi}
\psi\in L^s_{\mathrm{loc}}(\Omega) \ \ \text{with}\ \ s>\frac{n}{i_\infty(A)}
\end{equation}
and
\begin{equation}\label{requirements on phi}
\text{there is}\ \lambda>0\ \text{such that}\ A\circ b^{-1}(\lambda\phi)\in L_{\mathrm{loc}}^{r}(\Omega),\ \ \text{with}\ \ \begin{cases}
\displaystyle r>\frac{n}{i_\infty(A)}\frac{I_\infty(B)-1}{i_\infty(B)-1} &\mbox{if}\ i_\infty(B)>1\\
r=\infty &\mbox{if}\ i_\infty(B)=1
\end{cases}.
\end{equation}
By the properties of the indices at infinity, any modification of $A$ and $B$ near zero does not affect the assumptions on $\psi$ and $\phi$. This will be exploited in Lemma \ref{redefinition of functions and parameters}.

From \eqref{requirements on psi} and \eqref{requirements on phi}, the functions $\psi$ and $\phi$ belong to the local Orlicz spaces built upon the Young functions $\reallybigtilde{A_n}$ and $\widetilde{B}$, respectively. To see this, note that given any sufficiently smallTo see this, note that given any sufficiently small $\sigma\ge0$ \eqref{requirements on psi} implies
$$s>\frac{n}{i_\infty(A)-\sigma}>\Big(\frac{n(i_\infty(A)-\sigma)}{n-(i_\infty(A)-\sigma)}\Big)'.$$
Choosing $\sigma<\varepsilon$, where $\varepsilon$ is as in \eqref{lower growth on A by power}, \eqref{sobolev conjugate dominates the conjugate of the lower exponent} ensures that $t^s\gtrsim\reallybigtilde{A_n}(t)$ globally, which shows that
\begin{equation}\label{psi in the dual}
	\psi\in L^{\reallybigtilde{A_n}}_\mathrm{loc}(\Omega).
\end{equation}
Concerning $\phi$, if $i_\infty(B)=1$, \eqref{requirements on phi} directly gives $\phi\in L^\infty_\mathrm{loc}(\Omega)$, so the claim is trivial. When $i_\infty(B)>1$, \eqref{requirements on phi} implies that $\phi^{\frac{n}{i_\infty(B)-1}}$ belongs to $L^{\alpha}_{\mathrm{loc}}(\Omega)$ for some $\alpha>1$. Hence, by \eqref{lower index of B is smaller than n}, the same holds for $\phi^{\frac{i_\infty(B)}{i_\infty(B)-1}}$. Since \eqref{B dominates small powers} means that $\widetilde{B}$ grows slower than $t^{\beta'}$ at infinity for all $1\le\beta<i_\infty(B)$, it follows that
\begin{equation}\label{phi in the dual}
\phi\in L^{\widetilde{B}}_\mathrm{loc}(\Omega).
\end{equation}
\subsection{Definition of weak solution} Under the assumptions above, we can now define the notion of solution to Equation \eqref{equation}. As mentioned in the introduction, we consider solutions $u$ belonging to $W^1L^B_\mathrm{loc}(\Omega)$ and satisfying the equation in a (local) weak sense. Before giving a precise definition, we observe that, by \eqref{phi in the dual}, for every $u\in L^B_{\mathrm{loc}}(\Omega)$ and every $\Omega'\Subset\Omega$ it holds
$$\int\limits_{\Omega'} A(|u|)\dex\le \int\limits_{\Omega'} q B(|u|)+\phi|u|\dex<\infty.$$
Hence, any $u\in W^1L^B_\mathrm{loc}(\Omega)$ belongs to $W^1K^A_{\mathrm{loc}}(\Omega)$ and, by Theorem \ref{Orlicz-Sobolev embedding}, also to $L^{A_n}_\mathrm{loc}(\Omega)$.

Let $u\in W^1L^B_{\mathrm{loc}}(\Omega)$ and $ v \in W^1L^B(\Omega)$, and set $\Omega'\Def\supp\, v \Subset\Omega$. Using \eqref{growth}, the Cauchy-Schwarz and Young inequalities, and \eqref{inequality between A conjugate and A}, we obtain
\begin{align*}
\int\limits_\Omega|\A(x,\nabla u)\cdot \nabla  v |\dex
&\le \int\limits_{\Omega'}\widetilde{B}({b(|\nabla u|)})+ B(|\nabla v|)+ \phi |\nabla v|\dex\\
&\le \int\limits_{\Omega'\cap\{|\nabla u|<t_1\}}\widetilde{B}({b(t_1)})\dex+\int\limits_{\Omega'\cap\{|\nabla u|\ge t_1\}} |\nabla u|b(|\nabla u|)\dex+\int\limits_{\Omega'} B(|\nabla v|)+ \phi |\nabla v|\dex\\
&\le \widetilde{B}({b(t_1)})|\Omega'|+q\int\limits_{\Omega'}B(|\nabla u|)\dex+\int\limits_{\Omega'} B(|\nabla v|)+ \phi |\nabla v|\dex<\infty.
\end{align*}
The fact that
$\int\limits_{\Omega'} \phi|\nabla v|\dex$ is finite follows from \eqref{phi in the dual}.\\
Concerning the reaction term, by \eqref{growth on f} one has
$$\int\limits_\Omega |f(x,u)v|\dex\le \int\limits_{\Omega'} |\psi v|+\frac{E(|u|)}{|u|}|v| \dex.$$
Then, using \eqref{domination on E} together with \eqref{inequalities for A and a} and \eqref{inequality between A conjugate and A}, we get
\begin{align*}
\int\limits_\Omega |f(x,u)v|\dex&\le \int\limits_{\Omega'} |\psi v|\dex+\frac{E(t_0)}{t_0}\int\limits_{\Omega'}|v|\dex+\int\limits_{\Omega'} \frac{A_n(c|u|)}{|u|}|v|\dex\\
&\le\int\limits_{\Omega'} |\psi v | + \frac{E(t_0)}{t_0}|v|+A_n(c|u|)+A_n(c|v |)\dex<\infty,
\end{align*}
where to ensure the finiteness of the last integral one exploits \eqref{psi in the dual}.

In summary, for any $u\in W^1L^B_{\mathrm{loc}}(\Omega)$ and $ v \in W^1L^B(\Omega)$ with $\supp\, v \Subset\Omega$, both $\A(\cdot,\nabla u)\cdot \nabla  v $ and $f(\cdot,u) v $ belong to $L^1(\Omega)$. Therefore, the following definition of weak solution is well posed.
\begin{deff}\label{definition weak solution}
A function $u\in W^1L^B_{\mathrm{loc}}(\Omega)$ is a weak solution to \eqref{equation} if
\begin{equation}\label{weak formulation}
\int\limits_\Omega \A(x,\nabla u)\cdot \nabla  v \dex=\int\limits_\Omega f(x,u) v \dex
\end{equation}
for all $ v \in W^1L^B(\Omega)$ such that $\supp\, v \Subset\Omega$.
\end{deff}
\subsection{Statement of the main result}
Due to the structure of the problem, we adjust our assumptions according to the growth rate of $B$. If $B$ has linear growth (see Section \ref{Orlicz-Sobolev spaces}), no additional assumption is needed.

Suppose instead that $B$ exhibits super-linear growth. In this setting, we require that
\begin{gather}\label{assumption 2}
\reallybigtilde{A\circ b^{-1}}\ \text{satisfies the}\ \Delta_2\text{-condition near infinity}.
\end{gather}
If
\begin{equation}\label{first case main thm}
n\ge3\quad\text{and}\quad\int\limits^\infty\left(\frac{t}{A(t)}\right)^\frac{1}{n-2}\de t=\infty,
\end{equation}
then we impose the balance condition
\begin{equation}\label{main condition}
b\lesssim {\reallybigtilde{A_{n-1}}}^{-1}\circ A \quad\text{near infinity},
\end{equation}
where ${\reallybigtilde{A_{n-1}}}^{-1}$ denotes the generalized right-continuous inverse of $\reallybigtilde{A_{n-1}}$ (see \eqref{generalized right continuous inverse}).\\
On the other hand, in the complementary regime, that is, if
\begin{equation}\label{second case main thm}
\begin{cases}
n=2,\\
n\ge3\quad\text{and}\quad\displaystyle\int\limits^\infty\left(\frac{t}{A(t)}\right)^\frac{1}{n-2}\de t<\infty,
\end{cases}
\end{equation}
no further assumption is required.

Our main result is the following:
\begin{thm}\label{boundedness theorem} Let $A$, $B$ and $E$ be Young functions. Let $u\in W^{1}L^B_\mathrm{loc}(\Omega)$ be a weak solution to Equation \ref{equation}, satisfying the ellipticity and growth conditions \eqref{ellipticity}, \eqref{growth}, and \eqref{growth on f}. Assume that $B$ satisfies \eqref{integral at infinity is infinite} and \eqref{Delta2 assumption}, that $E$ satisfies \eqref{Delta2 on E} and \eqref{domination on E}, and that $\psi$ and $\phi$ satisfy \eqref{requirements on psi} and \eqref{requirements on phi}, respectively.
If $B$ has super-linear growth, assume additionally that \eqref{assumption 2} is satisfied and that either \eqref{second case main thm} is in force, or \eqref{first case main thm} is in force and \eqref{main condition} is satisfied.

Then, $u$ is locally bounded in $\Omega$.
\end{thm}
Some comments on the above statement are in order.
\begin{rmk}\label{Remark post thm}
$(a)$ The balance condition \eqref{main condition} and the dominance inequality \eqref{B dominates A for large t} are not contradictory. More precisely, the map $t\mapsto t {\reallybigtilde{A_{n-1}}}^{-1}(A (t))$ dominates $A$ near infinity, and hence may dominate $b$ as well. Indeed, since
$$ A\lesssim A_{n-1}\quad \text{near infinity},$$
by exploiting \eqref{inequality between inverses of A conjugate and A} we obtain
$${\reallybigtilde{A_{n-1}}}^{-1}\circ A (t)\gtrsim {\widetilde{A}}^{-1}\circ A (t)\approx \frac{A(t)}{t}\quad\text{near infinity}.$$
Moreover, since $A \not\approx A_{n-1}$ near infinity, it follows that
$${\reallybigtilde{A_{n-1}}}^{-1}\circ A (t)\gtrsim q\frac{A(t)}{t}\quad\text{near infinity}$$
for some $q>0$.

$(b)$ Let us compare Theorem \ref{boundedness theorem} with the known results on the local boundedness of local minimizers of variational integrals with unbalanced Orlicz growth obtained in \cite{CianchiSchaffner}. From \eqref{main condition} and the asymptotic relation $A\lesssim A_{n-1}$ near infinity, one has
\begin{equation}\label{to prove the variational}
\frac{B(t)}{t}\lesssim {\reallybigtilde{A_{n-1}}}^{-1}\circ A\lesssim {\reallybigtilde{A_{n-1}}}^{-1}\circ A_{n-1}\quad\text{near infinity}.
\end{equation}
Since, by \eqref{inequality between inverses of A conjugate and A},
$$t{\reallybigtilde{A_{n-1}}}^{-1}(A_{n-1}(t))\le 2A_{n-1}(t)\quad\text{for all}\ t\ge0,$$
\eqref{to prove the variational} implies
\begin{equation*}
B\lesssim A_{n-1}\quad\text{near infinity}.
\end{equation*}
Assuming \eqref{first case main thm} holds, this condition is precisely the one required in \cite{CianchiSchaffner} to prove the variational counterpart of Theorem \ref{boundedness theorem}. Consequently, our result is consistent with the variational framework.

$(c)$ If $B$ has linear-growth and \eqref{first case main thm} holds, then \eqref{main condition} is satisfied. Indeed, in this regime $A$ is equivalent to a linear function near infinity, and therefore
\begin{equation*}
{\reallybigtilde{A_{n-1}}}^{-1}\circ A\approx {\reallybigtilde{A_{n-1}}}^{-1}\ \text{near infinity}\ \text{and}\ b\ \text{is bounded}.
\end{equation*}
Moreover, by \eqref{lower growth on A by power}, $cA(t)\ge t$ for all $t\ge0$, thus the same argument used to prove \eqref{sobolev conjugate dominates the conjugate of the lower exponent} shows that
$$A_{n-1}(t)\gtrsim t^{\frac{n-1}{n-2}}\quad\text{globally}.$$
The condition in \eqref{main condition} then follows.

$(d)$ Thanks to the $\Delta_2$-condition on $B$, \eqref{main condition} is equivalent to
\begin{equation*}
	\alpha\frac{B(t)}{t}\lesssim {\reallybigtilde{A_{n-1}}}^{-1}\circ A\quad\text{near infinity}, 
\end{equation*}
for some $\alpha>0$. For later use, it is convenient to rewrite this as
\begin{equation*}
	\reallybigtilde{A_{n-1}}\lesssim A\circ b^{-1}\quad \text{near infinity},
\end{equation*}
which, roughly speaking, means that $A\circ b^{-1}$ does not grow too slowly in comparison with $A$.
\end{rmk}
Our approach to Theorem \ref{boundedness theorem} follows the lines of De Giorgi’s regularity result for linear equations with measurable coefficients. Inspired by \cite{CianchiSchaffner}, and as in any existing proof of local boundedness based on De Giorgi's technique, we derive a Caccioppoli-type inequality, i.e., an energy estimate on the super-level sets of the relevant solution $u$. For elliptic equations with natural growth as in \eqref{p growth}, a prototype of this inequality reads as:
\begin{equation}\label{Caccioppoli natural}
\int\limits_{\B_\rho(y)\cap\{u>k\}}|\nabla  (u-k)_+|^p\dex\lesssim_{n,p} \frac{1}{(\sigma-\rho)^p}\int\limits_{\B_\sigma(y)\cap\{u>k\}} (u-k)_+ ^p\dex+|\B_\sigma(y)\cap\{u>k\}|^\alpha
\end{equation}
for all concentric balls $\B_\rho(y)\Subset\B_\sigma(y)\Subset\Omega$. Taking any cut-off function $\eta$ associated with two such balls, and choosing
\begin{equation}\label{usual test function}
v= \eta (u-k)_+
\end{equation}
as a test function in the weak formulation of the problem, \eqref{Caccioppoli natural} follows from an absorption argument for the gradient terms, which relies on the growth conditions and Young's inequality.

When dealing with unbalanced Orlicz growths, further difficulties arise, due to both the different lower and upper bounds on $\A$ and the lack of homogeneity of non-power-type Young functions. Indeed, reproducing the argument outlined above generates unwanted integral terms on the right-hand side of the inequality. To overcome this difficulty, one needs a balance condition between $A$ and $B$. This requirement allows one to perform an absorption argument similar to the $p$-growth setting and allows for the application of an $n$-dimensional Sobolev-type inequality in a sharp integral form on the remaining terms.

We rely on the recent approach developed in \cite{CianchiSchaffner} for variational integrals with unbalanced Orlicz growths, which, in turn, refines the trial function optimization argument introduced in the $p,q$-growth framework in \cite{BellaSchaffner2020,HirschSchaffner}. This technique employs Sobolev inequalities only on $(n-1)$-dimensional spheres, rather than on $n$-dimensional balls. This allows us to formulate the balance inequality \eqref{main condition} in terms of the $(n-1)$-dimensional Sobolev conjugate of $A$ alone, without relying on the $n$-dimensional one, thereby yielding a sharper inequality. The strategy is based on Lemma \ref{lemma construction of cut off} below, which provides an optimized cut-off function in \eqref{usual test function}. Sobolev and Sobolev-Poincaré type inequalities in dimension $n$ come into play only when estimating terms depending explicitly on $u$. Since Sobolev conjugates are generally non-homogeneous, the corresponding inequalities necessarily involve a gradient term on both sides; hence, integrals involving both $u$ and $\nabla u$ appear on the right-hand side of the inequality.

The linear-growth case is handled separately. Indeed, as noted in part $(c)$ of Remark \ref{Remark post thm}, when $B$ has linear growth, $A\circ b^{-1}$ is too slow compared with $A_{n-1}$. On the other hand, since $b$ is bounded, no problematic integral terms appear on the right-hand side of the inequality, so an absorption argument is unnecessary. This makes the proof relatively short and elementary.

In order to better describe the content of Theorem \ref{boundedness theorem}, at the conclusion of this section we analyze some special instances of the relevant Young functions $A$ and $B$. We consider equations whose ellipticity is governed either by powers or by ``log-bumps'', namely Young functions of the form ``power-times-logarithm''. In the former case, we not only recover the available results but also improve and extend them in certain respects. The latter application provides an example where the ellipticity is associated with non-homogeneous Young functions.

\begin{xmp}\label{example powers}
First, we consider the standard case
$$A(t)=t^p\ \text{and}\ B\ \text{super-linear},$$
with $1\le p\le n$. Then, the assumptions \eqref{requirements on psi} and \eqref{requirements on phi} simply become
$$\psi\in L^s_\mathrm{loc}(\Omega)\ \ \text{with}\ \ s>\frac{n}{p}$$
and
$$ b^{-1}(\lambda\phi)\in L^\alpha_\mathrm{loc}(\Omega)\ \ \text{with}\ \ \begin{cases}
\displaystyle \alpha>n\frac{I_\infty(B)-1}{i_\infty(B)-1} &\mbox{if}\ i_\infty(B)>1,\\
\alpha=\infty &\mbox{if}\ i_\infty(B)=1.
\end{cases}$$
Moreover, for $1\le p<n$, the assumption \eqref{domination on E} is equivalent to
$$E(t)\lesssim t^{\frac{np}{n-p}}\quad\text{near infinity};$$
on the other hand, $A_n(t)\approx \ee^{t^{n'}}$ near infinity if $p=n$, hence \eqref{domination on E} is verified by any $E$ that satisfies the $\Delta_2$-condition near infinity.

Suppose $n\ge3$. If $1\le p<n-1$, then $A_{n-1}(t)\approx t^{\frac{(n-1)p}{n-1-p}}$ near infinity, and the assumption \eqref{main condition} is equivalent to
$$B(t)\lesssim t^{\frac{n}{n-1}p}=A(t^{n'})\quad\text{near infinity}.$$
In particular, when $B(t)=t^q$, with $1\le q \le n$, we see that Theorem \ref{boundedness theorem} recovers a result of \cite{CupiniMarcelliniMascolo23}. 
If $p=n-1$, then $A_{n-1}(t)\approx \ee^{t^{(n-1)'}}$ near infinity, hence ${\reallybigtilde{A_{n-1}}}^{-1}(t)\approx t(\log t)^{-\frac{1}{(n-1)'}}$ near infinity. Therefore, in this case \eqref{main condition} becomes
$$B(t)\lesssim t^n (\log t)^{-\frac{1}{(n-1)'}}\quad\text{near infinity}.$$
Thus, if Equation \eqref{equation} is $(n-1)$-elliptic, then in order for Theorem \ref{boundedness theorem} to imply the local boundedness of weak solutions, $L^n(\log L)^{-\frac{1}{(n-1)'}}$ is the smallest Orlicz space that can prescribe the upper growth condition on $\A$.

On the other hand, if either $n=2$ or $n\ge3$ and $p> n-1$, the second alternative condition in \eqref{second case main thm} is satisfied, hence no further condition on $B$ has to be imposed.
\end{xmp}
\begin{xmp}
We now turn our attention to Zygmund-type growths. Assume that $B$ has super-linear growth, and that
$$A(t)\approx t^p(\log t)^\alpha\quad\text{near infinity},$$
where either $1<p<n$ and $\alpha\in\R$, or $p=1$ and $\alpha\ge0$, or $p=n$ and $\alpha\le n-1$. These restrictions on the exponents $p$ and $\alpha$ (especially the last one) are required for $A$ to be a Young function fulfilling \eqref{integral at infinity of A}.

Under the above assumptions on $A$, its optimal $n$-dimensional Sobolev conjugate is equivalent near infinity to
$$\begin{cases}
t^\frac{np}{n-p}(\log t)^{\frac{n\alpha}{n-p}} &\mbox{if}\ 1\le p<n\\
\ee^{t^{\frac{n}{n-1-\alpha}}} &\mbox{if}\ p=n\ \text{and}\ \alpha<n-1\\
\ee^{\ee^{t^{n'}}} &\mbox{if}\ p=n\ \text{and}\ \alpha=n-1.
\end{cases}$$
Analogous considerations can be made for the $(n-1)$-dimensional Sobolev conjugate.

Since $i_\infty(A)=p$, the assumptions \eqref{requirements on psi} and \eqref{requirements on phi} rewrite as in Example \ref{example powers}. On the other hand, if $p<n$, \eqref{domination on E} becomes
$$E(t)\lesssim t^\frac{np}{n-p}(\log t)^{\frac{n\alpha}{n-p}}\quad \text{near infinity}.$$
If $p=n$, \eqref{domination on E} is verified by any Young function satisfying the $\Delta_2$-condition near infinity.

Note that \eqref{first case main thm} is equivalent to
\begin{equation}\label{first case log bumps}
n\ge3\quad\text{and}\quad\begin{cases}
	p<n-1,\\
	p=n-1\ \ \text{and}\ \ \alpha\le n-2.
\end{cases}
\end{equation}
Then, by Theorem \eqref{boundedness theorem}, every weak solution to \eqref{equation} is locally bounded provided that
$$\begin{cases}
n\ge2,\\
n\ge3 \ \ \text{and}\ \ p>n-1,\\
n\ge 3,\ \ p=n-1\ \ \text{and}\ \ \alpha> n-2.
\end{cases}$$

Assume that \eqref{first case log bumps} is in force. In this case, we need to impose the balance condition \eqref{main condition}. If $p<n-1$, one has
$$A_{n-1}(t)\approx t^\frac{(n-1)p}{n-1-p}(\log t)^{\frac{(n-1)\alpha}{n-1-p}}\quad\text{near infinity},$$
hence, exploiting \eqref{inequality between inverses of A conjugate and A},
$${\reallybigtilde{A_{n-1}}}^{-1}(t)\approx t^{\frac{np-n+1}{(n-1)p}}(\log t)^{\frac{\alpha}{p}}\quad\text{near infinity}.$$
Then, \eqref{main condition} rewrites as
\begin{equation*}
B(t)\lesssim t^{n'p}(\log t)^{n'\alpha}\quad \text{near infinity}.
\end{equation*}
If $p=n-1$ and $\alpha\le n-2$, $A_{n-1}(t)$ is equivalent near infinity to
$$\begin{cases}
\ee^{t^{\frac{n-1}{n-2-\alpha}}} &\mbox{if}\ \alpha<n-2\\
\ee^{\ee^{t^{(n-1)'}}} &\mbox{if}\ \alpha=n-2.
\end{cases}$$
Therefore,
$${\reallybigtilde{A_{n-1}}}^{-1}(t)\approx\begin{cases}
t(\log t)^{-\frac{n-2-\alpha}{n-1}} &\mbox{if}\ \alpha<n-2\\
t(\log\log t)^{-\frac{1}{(n-1)'}} &\mbox{if}\ \alpha=n-2
\end{cases}\quad\text{near infinity},$$
hence \eqref{main condition} becomes
$$B(t)\lesssim\begin{cases}
t^{n}(\log t)^{\alpha-\frac{n-2-\alpha}{n-1}} &\mbox{if}\ \alpha<n-2\\
t^n(\log t)^{\alpha}(\log\log t)^{-\frac{1}{(n-1)'}} &\mbox{if}\ \alpha=n-2
\end{cases}\quad\text{near infinity}.$$
\end{xmp}

\section{Technical lemmas}\label{Technical lemmas}
All the conditions stated in Section \ref{main results} are invariant under the replacement of $A$, $B$, and $E$ with equivalent Young functions near infinity. Indeed, in these conditions, only the behavior of the relevant Young functions --- or of their left derivatives --- \emph{for large arguments} matters. We aim to show that it is possible to redefine $A$, $B$, and $E$ near zero, and correspondingly modify the parameters $R$, $\lambda$, $\phi$, and $\psi$, in such a way that all the assumptions of Theorem \ref{boundedness theorem} hold globally, not just near infinity. The result is formulated in the following lemma:
\begin{lemma}\label{redefinition of functions and parameters}
Let the assumptions of Theorem \ref{boundedness theorem} be satisfied. Then the following statements hold:\\
$\boldsymbol{(i)}$ \emph{\textbf{Existence of modified Young functions and measurable functions.}} There exist Young functions $\widehat{A}$, $\widehat{B}$, $\widehat{E}$, and non-negative measurable functions $\widehat{\phi}$, $\widehat{\psi}$, such that:\\
$(1)$ $\widehat{A}=A$, $\widehat{B}=B$, and $\widehat{E}=E$ near infinity.\\
$(2)$ The exponents $s$ and $r$ satisfy
\begin{equation}
s>\frac{n}{i_\infty(\widehat{A})}\quad\text{and}\quad\begin{cases}
	\displaystyle r>\frac{n}{i_\infty(\widehat{A})}\frac{I_\infty(\widehat{B})-1}{i_\infty(\widehat{B})-1} &\mbox{if}\ i_\infty(\widehat{B})>1,\\
	r=\infty &\mbox{if}\ i_\infty(\widehat{B})=1.
\end{cases}\label{summability inequality 0 redefinition}
\end{equation}
$(3)$ $\widehat{A}(t), \widehat{B}(t)>0$ for all $t>0$, and
\begin{equation}
\int\limits_0 \bigg(\frac{t}{\widehat{A}(t)}\bigg)^{\frac{1}{n-1}}\de t<\infty.\label{finiteness integral of A redefinition}
\end{equation}
$(4)$ For \Ae $x\in\Omega$, the maps $\A(x,\cdot)$ and $f(x,\cdot)$ are controlled by $\widehat{A},\widehat{B},\widehat{E},\widehat{\psi},\widehat{\phi}$ in the following sense:
\begin{gather}
|\A(x,\xi)|\le \widehat{b}(|\xi|)+\widehat{\phi}(x)\quad\text{for all}\ \xi\in\R^n,\nonumber\\
\A(x,\xi)\cdot\xi\ge \widehat{A}(|\xi|)\quad\text{for all}\ \xi\in\R^n,\ |\xi|\ge\widehat{R},\nonumber\\
|f(x,z)|\le \widehat{\psi}(x)+\frac{\widehat{E}(|z|)}{|z|} \quad\text{for all}\ z\in\R.\nonumber
\end{gather}
$\boldsymbol{(ii)}$ \emph{\textbf{Choice of constants and growth conditions.}} There exist constants $\widehat{L},\widehat{R}\ge1$, $\widehat{\lambda}>0$, and exponents $p\ge1$ and $m>pn'$, such that:\\
$(1)$ the exponent $p$ satisfies
\begin{equation}
\begin{cases}
	1<p<i_\infty(\widehat{A}) &\mbox{if}\ i_\infty(\widehat{A})>1\\
	p=1 &\mbox{if}\ i_\infty(\widehat{A})=1
\end{cases}\qquad\text{and}\qquad s,r>\frac{n}{p}.\label{summability inequality redefinition}
\end{equation}
$(2)$ The functions $\widehat{\psi}$ and $\widehat{\phi}$ satisfy
\begin{equation}
\widehat{\psi}\in L^s_{\mathrm{loc}}(\Omega),\quad \widehat{A}\circ\widehat{b}^{-1}(\widehat{\lambda}\widehat{\phi})\in L^r_{\mathrm{loc}}(\Omega).\label{requirements on phi redefinition}
\end{equation}
$(3)$ The Young functions satisfy the following growth conditions for all $t\ge0$:
\begin{gather}
t^p\le \widehat{L}\widehat{A}(t),\label{A hat dominates curve globally}\\
\widehat{B}(2t)\le \widehat{L}\widehat{B}(t),\label{B Delta2 redefinition}\\
\widehat{E}(2t)\le\widehat{L}E(t)\quad\text{and}\quad \widehat{E}(t)\le{\widehat{A}_n(\widehat{L}t)}.\label{growth inequalities for E}
\end{gather}
$(4)$ Depending on the growth of $\widehat{B}$:
\begin{itemize}
\item If $\widehat{B}$ has linear growth, then its left-derivative is bounded by $\widehat{L}$:
\begin{equation}\label{b is bounded}
	\widehat{b}(t)\le\widehat{L}\quad \text{for all}\ t\in[0,\infty).
\end{equation}
\item If $\widehat{B}$ has super-linear growth, then $\reallybigtilde{\widehat{A}\circ\widehat{b}^{-1}}$ satisfies the $\Delta_2$-condition globally; equivalently,
\begin{equation}
	\reallybigtilde{\widehat{A}\circ\widehat{b}^{-1}}(\mu t)\le \mu^m\reallybigtilde{\widehat{A}\circ\widehat{b}^{-1}}(t)\ \ \text{for all}\ t\ge0\ \text{and all}\ \mu\ge1.\label{domination by power on Aob^-1}
\end{equation}
Further, if \eqref{first case main thm} and \eqref{main condition} are in effect,
\begin{equation}\label{main condition redefinition}
	\reallybigtilde{\widehat{A}\circ\widehat{b}^{-1}}(t)\le \widehat{A}_{n-1}(\widehat{L}t)\quad\text{for all}\ t\ge0,
\end{equation}
while if \eqref{second case main thm} is in effect,
\begin{equation}\label{second condition redefinition}
	\reallybigtilde{\widehat{A}\circ\widehat{b}^{-1}}(t)\le \widehat{L} t^m\quad\text{for all}\ t\ge0.
\end{equation}
\end{itemize}
\end{lemma}
Before presenting the proof, we state a useful lemma illustrating how to construct super-linear Young functions with prescribed behavior near zero and at infinity. The proof will make use of a classical result from measure theory; for details, we refer the reader to Theorem 1.16 in \cite{folland}:
\begin{prop}\label{prop derivative of increasing functions}
Let $r>0$. If $\varphi\colon(r,\infty)\to\R$ is increasing and right-continuous, then there exists a positive Borel measure $\mu$ on $(r,\infty)$ such that
$$\mu((a,b])=\varphi(b)-\varphi(a)\quad\text{for all}\ r<a<b<\infty.$$
\end{prop}
Given Young functions $\Phi,\Psi\colon[0,\infty)\to[0,\infty)$ and numbers $\tau_1>\tau_0>0$, we define $\mathcal{J}_{\tau_0,\tau_1}[\Phi,\Psi]\colon[0,\infty)\to[0,\infty)$ by
$$\J_{\tau_0,\tau_1}[\Phi,\Psi](t)\Def\begin{cases}
\Phi(t) &\mbox{if}\ 0\le t< \tau_0\\
\displaystyle\frac{\tau_1-t}{\tau_1-\tau_0}\Phi(\tau_0)+\frac{t-\tau_0}{\tau_1-\tau_0}\Psi(\tau_1) &\mbox{if}\ \tau_0\le t< \tau_1\\
\Psi(t) &\mbox{if}\ t\ge \tau_1
\end{cases}\qquad\text{for}\ t\ge0.$$
In other words, we interpolate linearly between $\Phi$ and $\Psi$ on $[\tau_0,\tau_1]$.
\begin{lemma}\label{lemma joining convex functions}
Let $\Phi,\Psi\colon[0,\infty)\to[0,\infty)$ be Young functions, and assume that $\Psi$ has super-linear growth. Then, for every $\tau_0>0$ there exists $T=T(\tau_0)\ge\tau_0$ such that, for any $\tau_1> T$, $\J_{\tau_0,\tau_1}[\Phi,\Psi]$ is a Young function.
\end{lemma}
\begin{proof}
Fix $\tau_0>0$. For each $\tau_1>\tau_0$$\tau_1>\tau_0$, the function $\J_{\tau_0,\tau_1}[\Phi,\Psi]$ is continuous by definition. To determine a suitable $T$, we compare the slope of the line joining $(\tau_0,\Phi(\tau_0))$ and $(t,\Psi(t))$ with the slopes of $\Phi$ and $\Psi$ at $\tau_0$ and $t$, respectively. Our aim is to show that, for all sufficiently large $t$, this interpolating slope lies between the corresponding one-sided derivatives of $\Phi$ and $\Psi$; that is,
\begin{equation*}
\Phi'_-(\tau_0)\le \frac{\Psi(t)-\Phi(\tau_0)}{t-\tau_0}\le \Psi'_+(t).
\end{equation*}
This is equivalent to
\begin{equation}\label{claim in convex junction lemma}
\Phi'_-(\tau_0)\le \frac{\Psi(t)-\Psi(\tau_0)}{t-\tau_0}+\frac{\Psi(\tau_0)-\Phi(\tau_0)}{t-\tau_0}\le \Psi'_+(t).
\end{equation}
Since $\Psi$ is super-linear, its slopes diverge, namely
$$\lim_{t\to\infty}\frac{\Psi(t)-\Psi(\tau)}{t-\tau}=\infty\qquad\text{for all}\ \tau>0.$$
Thus, the left-hand inequality in \eqref{claim in convex junction lemma} holds for all $t$ large enough.

To prove the right-hand inequality in \eqref{claim in convex junction lemma}, it suffices to show that
\begin{equation}\label{eta diverges}
\eta(t)\Def(t-\tau_0)\Psi'_+(t)-(\Psi(t)-\Psi(\tau_0))\to\infty\quad\text{as}\ t\to\infty.
\end{equation}
Observe first that
$$\eta(t)=\int\limits_{\tau_0}^t\Psi'_+(t)-\Psi'_+(s)\de s.$$
Let $t>\tau_0$. Since $\Psi'_+\colon(\tau_0,\infty)\to\R$ is increasing and right-continuous, by Proposition \ref{prop derivative of increasing functions} there exists a positive Borel measure $\mu$ on $(\tau_0,\infty)$ such that
$$\Psi'_+(y)-\Psi'_+(s)=\mu((s,y])\quad\text{for}\ \tau_0<s<y<\infty.$$
Using this representation, we may rewrite $\eta(t)$ as
\begin{align*}
\eta(t)=\int\limits_{\tau_0}^t\bigg(\int\limits_{(\tau_0,t]}\chi_{(s,t]}(y)\de \mu(y)\bigg)\de s= \int\limits_{(\tau_0,t]}\bigg(\int\limits_{\tau_0}^t\chi_{(s,t]}(y)\de s\bigg)\de \mu(y)=\int\limits_{(\tau_0,t]}y-\tau_0\de \mu(y),
\end{align*}
where the second equality follows from Fubini's theorem. Let $0<\varepsilon<t-\tau_0$. Since $y-\tau_0\ge\varepsilon$ for all $y\in(\tau_0+\varepsilon,t]$, we obtain the lower bound
$$\eta(t)\ge\int\limits_{(\tau_0+\varepsilon,t]}y-\tau_0\de \mu(y)\ge \varepsilon\mu((\tau_0+\varepsilon,t])=\varepsilon(\Psi'_+(t)-\Psi'_+(\tau_0+\varepsilon)).$$
In summary, we have shown that
$$\eta(t)\ge \varepsilon(\Psi'_+(t)-\Psi'_+(\tau_0+\varepsilon))\quad\text{for all}\ t>\tau_0.$$
Therefore, \eqref{eta diverges} follows from the super-linearity of $\Psi$. This yields the desired right-hand inequality in \eqref{claim in convex junction lemma} for all sufficiently large $t$, concluding the proof.
\end{proof}
Now we are ready to prove Lemma \ref{redefinition of functions and parameters}.
\begin{proof}[Proof of Lemma \ref{redefinition of functions and parameters}]
Recall from Section \ref{Orlicz-Sobolev spaces} that equivalence of Young functions near infinity implies equality of their Matuszewska--Orlicz indices at infinity. Hence, if one considers any Young functions $\widehat{A}$ and $\widehat{B}$ that are equivalent near infinity to $A$ and $B$, respectively, then the assumptions \eqref{requirements on psi} and \eqref{requirements on phi} are sufficient to ensure that \eqref{summability inequality 0 redefinition} holds. Furthermore, there exists an exponent $p\ge1$, depending only on $A$, for which \eqref{summability inequality redefinition} holds; importantly, this $p$ does not rely on the specific choice of $\widehat{A}$.\\
Given $p\ge1$ such that \eqref{summability inequality redefinition} is satisfied, there exists $c=c(A)>0$ and $t_2\ge0$ such that
\begin{equation}\label{definition of t3}
cA(t)\ge t^p\quad\text{for  all}\ t\ge t_2.
\end{equation}
If $B$ is super-linear and \eqref{main condition} is satisfied, let $L>0$ and $t_3\ge0$ be such that
\begin{equation}\label{definition of L}
b(t)\le {\reallybigtilde{A_{n-1}}}^{-1}(A(Lt))\quad \text{for all}\ t\ge t_3.
\end{equation}
Pick $t_4\ge0$ such that $A(t),B(t)>0$ for $t> t_4$. Let $t_5$ be the lower bound for $t$ in the definition of the $\Delta_2$-condition near infinity for the functions $E$ and $\reallybigtilde{A\circ b^{-1}}$. Set
$$\tau\Def\max\{R,t_0,t_1,t_2,t_3,t_4,t_5,1\},$$
where $R$, $t_0$, and $t_1$ are the constants appearing in \eqref{ellipticity}, \eqref{domination on E}, and \eqref{estimate of B with power}, respectively.

\emph{Step $1$: construction of $\widehat{A}$.} We distinguish between linear and super-linear behavior. If $A$ is super-linear, by Lemma \ref{lemma joining convex functions} there exists $R_1>\tau$ such that $\J_{\tau,R_1}[t^p,A(\cdot)]$ is convex. Then, we define $\widehat{A}\colon[0,\infty)\to[0,\infty)$ by
$$\widehat{A}(t)\Def\begin{cases}
	t^p &\mbox{if}\ 0\le t< \tau\\
	\displaystyle\frac{{R_1}-t}{{R_1}-\tau}\tau^p+\frac{t-\tau}{{R_1}-\tau}A({R_1}) &\mbox{if}\ \tau\le t< R_1\\
	A(t) &\mbox{if}\ t\ge R_1.
\end{cases}$$
Clearly, $\widehat{A}(t)=0$ iff $t=0$, and $\widehat{A}=A$ near infinity. By \eqref{definition of t3}, $cA(t)\ge t^p$ for all $t\ge R_1$. Moreover,
$$\widehat{A}(t)\ge \frac{\tau^p}{R_1^p}t^p$$
for $t\in[\tau,R_1)$. Since $\widehat{A}(t)=t^p$ in $[0,\tau)$, the inequality in \eqref{A hat dominates curve globally} is satisfied, with $\widehat{L}\ge \max{\big\{c,1,\frac{R_1^p}{\tau^p}\big\}}$. Then, \eqref{finiteness integral of A redefinition} easily follows from $p<n$.

On the contrary, when $A(t)/t\to l<\infty$ as $t\to\infty$, one has $i_\infty(A)=1$, hence $p=1$ too. Then, we define $\widehat{A}\colon[0,\infty)\to[0,\infty)$ by
$$\widehat{A}(t)\Def\begin{cases}
	\displaystyle\frac{A(\tau)}{\tau}t &\mbox{if}\ 0\le t< \tau\\
	A(t) &\mbox{if}\ t\ge \tau.
\end{cases}$$
As in the previous case, $\widehat{A}$ satisfies both \eqref{finiteness integral of A redefinition} and \eqref{A hat dominates curve globally}.

For later use, let us emphasize that, being locally Lipschitz continuous and strictly increasing, $\widehat{A}$ is invertible. We also observe that, if \eqref{first case main thm} is satisfied, then there exist $c>0$ and $0<\tau_0\le\tau$ such that
\begin{equation}\label{d-1 conjugate near zero redefinition}
	\widehat{A}_{n-1}(t)= ct^{\frac{(n-1)p}{n-1-p}}\quad \text{for}\ t\in[0,\tau_0].
\end{equation}

\emph{Step $2$: construction of $\widehat{B}$, super-linear case.} Suppose that $B$ has super-linear growth. First, assume that \eqref{first case main thm} is in force and \eqref{main condition} is satisfied. Let
$${\reallybigtilde{\widehat{A}_{n-1}}}^{-1}$$be the generalized right-continuous inverse of $\reallybigtilde{\widehat{A}_{n-1}}$, and define $N\colon[0,\infty)\to[0,\infty)$ as
$$N(t)=\int\limits_0^t {\reallybigtilde{\widehat{A}_{n-1}}}^{-1}\circ\widehat{A}(s)\de s\quad\text{for all}\ t\ge0.$$
Since $\widehat{A}\approx A$ near infinity, the same happens for $\widehat{A}_{n-1}$ and $A_{n-1}$. Then, there exist $\tau_1>0$ and $L_1>0$ such that
$${\reallybigtilde{{A}_{n-1}}}^{-1}(A (t))\le{\reallybigtilde{\widehat{A}_{n-1}}}^{-1}(\widehat{A}(L_1t))\quad \text{for}\ t\ge \tau_1.$$
Reasoning as in the construction of $\widehat{A}$, there exists $R_2>\max\{R_1,\tau_1\}$ such that $\J_{\tau_0,R_2}[N,B]$ is a Young function, where $\tau_0$ is defined by \eqref{d-1 conjugate near zero redefinition}. Then, we let $\widehat{B}\colon[0,\infty)\to[0,\infty)$ be given by
$$\widehat{B}(t)\Def\begin{cases}
	N(t) &\mbox{if}\ 0\le t< \tau_0\\
	\displaystyle\frac{R_2-t}{R_2-\tau_0}N(\tau_0)+\frac{t-\tau_0}{R_2-\tau_0}B(R_2) &\mbox{if}\ \tau_0\le t< R_2\\
	B(t) &\mbox{if}\ t\ge R_2.
\end{cases}$$
Clearly, $\widehat{B}$ vanishes only at $0$ and coincides with $B$ near infinity. Let us show that the inequality in \eqref{main condition redefinition} holds for a sufficiently large $\widehat{L}$. For $t\in[0,\tau_0)\cup[R_2,\infty)$, by \eqref{main condition} and the definition of $N$, it suffices to choose $\widehat{L}\ge\max\{1,L\}$, where $L$ is the constant in \eqref{definition of L}. For $t\in[\tau_0,R_2)$, we have
$$\widehat{b}(t)=\frac{B(R_2)-N(\tau_0)}{R_2-\tau_0}\le b(R_2) \le {\reallybigtilde{{A}_{n-1}}}^{-1}({A}(R_2))\le {\reallybigtilde{\widehat{A}_{n-1}}}^{-1}\circ\widehat{A}(L_1R_2)\le {\reallybigtilde{\widehat{A}_{n-1}}}^{-1}\circ\widehat{A}((L_1R_2/\tau_0) t),$$
where the first inequality follows from the convexity of $\widehat{B}$, the second from the dominance condition in \eqref{main condition}, and the third from the definition of $\tau_1$ and the lower bound on $t$. Hence, the inequality in \eqref{main condition redefinition} is satisfied provided $\widehat{L}\ge\{1,L,L_1R_2/\tau_0\}$.\\
To prove that \eqref{domination by power on Aob^-1} holds, we claim that the Young function $G\Def\reallybigtilde{\widehat{A}\circ\widehat{b}^{-1}}$ satisfies the $\Delta_2$-condition globally. This is equivalent to showing that there exists a constant $c=c(A,b)>0$ such that
\begin{equation}\label{fraction Delta2 Aob^-1 redefinition}
	\frac{tg(t)}{G(t)}\le c\quad\text{for}\ t>0.
\end{equation}
Since $G$ and $g$ are increasing, it suffices to verify \eqref{fraction Delta2 Aob^-1 redefinition} for $t$ near infinity and near zero. On the one hand, near infinity we have $\widehat{A}=A$ and $\widehat{b}=b$, so that, by \eqref{assumption 2}, \eqref{fraction Delta2 Aob^-1 redefinition} clearly holds for large $t$, with a suitable choice of $c$. On the other hand, by \eqref{d-1 conjugate near zero redefinition}, for $t$ close to $0$ one has
$$G(t)=ct^{\frac{(n-1)p}{n-1-p}}$$
for some constant $c=c(A,n,p)$. Hence, when $t$ is small,
$$\frac{tg(t)}{G(t)}=\frac{(n-1)p}{n-1-p},$$
so that \eqref{fraction Delta2 Aob^-1 redefinition} is also satisfied near zero.\\
Let us show that \eqref{B Delta2 redefinition} holds. To this end, we need to prove that
\begin{equation}\label{fraction Delta2 B redefinition}
	\frac{t\widehat{b}(t)}{\widehat{B}(t)}\le c\quad\text{for}\ t>0.
\end{equation}
Since $\widehat{B}=B\in\Delta_2$ near infinity, \eqref{fraction Delta2 B redefinition} is automatically satisfied for large $t$. It thus remains to check the condition near zero. Using \eqref{d-1 conjugate near zero redefinition}, one finds
$$\quad\widehat{B}(t)=c_1t^{\frac{n}{n-1}p}\quad\text{for all}\ t\in[0,\tau_0]$$
for some $c_1>0$. Consequently,
$$\frac{t\widehat{b}(t)}{\widehat{B}(t)}=\frac{n}{n-1}p\quad\text{near zero},$$
which establishes \eqref{B Delta2 redefinition}.

Now, assume that \eqref{second case main thm} is in force. By \eqref{assumption 2}, $\reallybigtilde{A\circ b^{-1}}$ satisfies the $\Delta_2$-condition near infinity. Recalling the definition of $t_5$, there exist $m\ge1$ and $c>0$ such that
$$\reallybigtilde{A\circ b^{-1}}(t)\le ct^m\quad\text{for all}\ t\ge t_5.$$
Without loss of generality, we may assume that
$$m>pn'.$$
Set
$$Q(t)\Def\int\limits_0^t {\widehat{A}}\big(s^{1/m'}\big)\de s\quad\text{for all}\ t\ge0$$
and choose $R_2'> R_1$ so that $\J_{\tau,R'_2}[Q,B]$ is a Young function. Defining $\widehat{B}\colon[0,\infty)\to[0,\infty)$ by
$$\widehat{B}(t)\Def \J_{\tau,R'_2}[Q,B](t)\quad\text{for all}\ t\ge0,$$
one argues in the same way as in the previous case to conclude that \eqref{domination by power on Aob^-1}, \eqref{B Delta2 redefinition} and \eqref{second condition redefinition} also hold in the present setting.

\emph{Step $3$: construction of $\widehat{B}$, linear case.} We now consider the case where $B$ has linear growth. By Remark \ref{A has a polynomial growth}, $A$ has also linear growth. Then, we define $\widehat{B}\colon[0,\infty)\to[0,\infty)$ by
$$\widehat{B}(t)\Def\begin{cases}
	\displaystyle\frac{B(\tau)}{\tau}t &\mbox{if}\ 0\le t< \tau\\
	B(t) &\mbox{if}\ t\ge\tau.
\end{cases}$$
Since
$$\frac{t\widehat{b}(t)}{\widehat{B}(t)}=1\quad\text{for $t$ near zero},$$
$\widehat{B}$ satisfies the $\Delta_2$-condition globally, i.e., the inequality in \eqref{B Delta2 redefinition} is satisfied for all $t\ge0$ and for a sufficiently big constant $\widehat{L}$. The inequality in \eqref{b is bounded} follows easily.

\emph{Step $4$: construction of $\widehat{E}$.}
Let
$$R_3=\begin{cases}
	R_2 & \mbox{if}\ \text{\eqref{first case main thm} is in effect}\\
	R_2' & \mbox{if}\ \text{\eqref{second case main thm} is in effect}\\
	\tau & \mbox{if}\ B\ \text{has linear growth}.
\end{cases}$$
Since $\widehat{A}\approx A$ near infinity, we can choose $\tau_2,L_2>0$ such that
$$A_n(t)\le\widehat{A}_n(L_2t)\quad\text{for all}\ t\ge\tau_2.$$
Take $R_4>\max\{R_3,\tau_2\}$ such that $\J_{\tau,R_4}[\widehat{A}_n,E]$ is a Young function. We define $\widehat{E}\colon[0,\infty)\to[0,\infty)$ analogously to $\widehat{B}$, by setting
$$\widehat{E}(t)\Def \J_{\tau,R_4}[\widehat{A}_n,E](t)=\begin{cases}
\widehat{A}_n(t) &\mbox{if}\ 0\le t< \tau\\
\displaystyle\frac{R_4-t}{R_4-\tau}\widehat{A}_n(\tau)+\frac{t-\tau}{R_4-\tau}E(R_4) &\mbox{if}\ \tau\le t< R_4\\
E(t) &\mbox{if}\ t\ge R_4
\end{cases}\quad\text{for all}\ t\ge0.$$
Then, \eqref{growth inequalities for E} follows adapting the argument of Step $2$, exploiting the assumption \eqref{domination on E}, and taking $\widehat{L}$ big enough.

\emph{Step $5$: conclusion.} Finally, we prove that the structure of the problem remains unchanged. Let
$$\widehat{R}\Def R_4.$$
Owing to the previous construction and the monotonicity of $b$ and of the function $t\mapsto\frac{E(t)}{t}$, we have
\begin{gather*}
	b(t)\le \widehat{b}(t)+b(\widehat{R})\quad\text{for all}\ t\ge0,\\
	\widehat{A}(t)=A(t)\quad\text{for all}\ t\ge\widehat{R},\\
	\frac{E(t)}{t}\le \frac{\widehat{E}(t)}{t}+\frac{E(\widehat{R})}{\widehat{R}}\quad\text{for all}\ t\ge0.
\end{gather*}
Then, for \Ae $x\in\Omega$, \eqref{ellipticity}, \eqref{growth}, and \eqref{growth on f} yield:
\begin{gather*}
	|\A(x,\xi)|\le b(|\xi|)+\phi(x)\le \widehat{b}(|\xi|)+b(\widehat{R})+\phi(x)\quad\text{for all}\ \xi\in\R^n,\\
	\A(x,\xi)\cdot\xi\ge A(|\xi|)=\widehat{A}(|\xi|)\quad\text{for all}\ \xi\in\R^n,\ |\xi|\ge\widehat{R},\\
	|f(x,z)|\le \psi(x)+\frac{E(|z|)}{|z|}\le \psi(x)+\frac{E(\widehat{R})}{\widehat{R}}+\frac{\widehat{E}(|z|)}{|z|}\quad\text{for all}\ z\in\R.\\
\end{gather*}
Set $$\widehat{\phi}(x)=\phi(x)+b(\widehat{R})\quad\text{and}\quad\widehat{\psi}(x)=\psi(x)+\frac{E(\widehat{R})}{\widehat{R}}.$$
The $s$-integrability of $\widehat{\psi}$ is an immediate consequence of \eqref{requirements on psi}. The conclusion then follows once we find $\widehat{\lambda}$ such that \eqref{requirements on phi redefinition} is satisfied. If $r=\infty$, the latter is equivalent to requiring that $\widehat{\phi}$ is locally bounded, which is easily seen from \eqref{requirements on phi}. Now, assume $r<\infty$ and set $\widehat{\lambda}=\lambda/2$, where $\lambda$ comes from \eqref{requirements on phi}. Let $t'>0$ be such that
$$\widehat{b}^{-1}(t)=b^{-1}(t)\quad\text{for all}\ t\ge t'.$$
For any $\Omega'\Subset\Omega$, one has
\begin{align*}
	\int\limits_{\Omega'} |\widehat{A}\circ\widehat{b}^{-1}(\widehat{\lambda}\widehat{\phi})|^r\dex&\le \int\limits_{\Omega'\cap\{\phi\le b(t')\}} \big|\widehat{A}\circ\widehat{b}^{-1}\big(\lambda b(t')\big)\big|^r\dex+\int\limits_{\Omega'\cap\{\phi> b(t')\}} |\widehat{A}\circ\widehat{b}^{-1}(\lambda\phi)|^r\dex\\
	&\le \big|\widehat{A}\circ\widehat{b}^{-1}\big(\lambda b(t')\big)\big|^r|\Omega'|+\int\limits_{\Omega'} |\widehat{A}\circ\widehat{b}^{-1}(\lambda\phi)|^r\dex.
\end{align*}
By the definition of $\widehat{A}$ and the choice of $t'$, we have $\widehat{A}\circ\widehat{b}^{-1}(t)={A}\circ{b}^{-1}(t)$ for all $t\ge t'$. It then follows from \eqref{requirements on phi} that $|\widehat{A}\circ\widehat{b}^{-1}(\lambda\phi)|^r\in L^1(\Omega')$. Combining this with the previous chain of inequalities, we conclude that
$$\widehat{A}\circ\widehat{b}^{-1}(\widehat{\lambda}\widehat{\phi})\in L^r(\Omega'),$$
as desired.
\end{proof}
Note that, by \eqref{domination by power on Aob^-1},
\begin{equation*}
\reallybigtilde{A\circ b^{-1}}(\mu \cdot)\le  \Phi_m(\mu) \reallybigtilde{A\circ b^{-1}}(\cdot)\quad\text{for all}\ \mu\ge0,
\end{equation*}
where $ \Phi_m\colon[0,\infty)\to[0,\infty)$ is defined as
$$ \Phi_m(\mu)=\begin{cases}
\mu &\mbox{if}\ 0\le \mu\le 1\\
\mu^m &\mbox{if}\ \mu>1.
\end{cases}$$
Given $u\in W^1K^A_{\mathrm{loc}}(\mathbb{B}_1)$, we set
$$\mathcal{F}_{\rho,\sigma}(u)=\int\limits_{\mathbb{B}_\sigma\setminus \mathbb{B}_\rho}A(|u|)+A(|\nabla u|)\dex.$$
A key ingredient in the proof of Theorem \ref{boundedness theorem} is the following result. Its proof proceeds as in Lemma $5.1$ of \cite{CianchiSchaffner}, with Theorem \eqref{Orlicz-Sobolev on spheres} used in place of Theorem $\mathrm{C}$ therein.
\begin{lemma}\label{lemma construction of cut off}
Let $A$ and $B$ be as in the statement of Theorem \ref{boundedness theorem}, and assume that $B$ has super-linear growth. Let $m>1$ be as in \eqref{domination by power on Aob^-1}, and let $0<\rho<\sigma<1$.

$(i)$ Suppose that \eqref{first case main thm} holds and that \eqref{main condition} is in effect. Then, for every $u\in W^1K^A(\B_1)$ there exists a function $\eta\in W^{1,\infty}_0(\B_1)$ satisfying
\begin{equation}\label{definition of eta}
	0\le\eta\le1,\quad \eta=1\ \text{in}\ \B_\rho,\quad \eta=0\ \text{in}\ \B_1\setminus\B_\sigma,\quad \|\nabla \eta\|_{L^\infty(\B_1)}\le \frac{2}{\sigma-\rho},
\end{equation}
such that
\begin{equation}\label{first estimate lemma}
	\int\limits_{\B_1} \reallybigtilde{A\circ b^{-1}}(|u\nabla \eta|)\dex\le c \Phi_ m\Big(\frac{\gamma \mathcal{F}_{\rho,\sigma}(u)^{\frac{1}{n-1}}}{(\sigma-\rho)^{n'}\rho}\Big)\mathcal{F}_{\rho,\sigma}(u)
\end{equation}
for some positive constant $c=c(n,m,L)$.

$(ii)$ Suppose that \eqref{second case main thm} is in effect. Then, for every $u\in W^1K^A(\B_1)$ there exists a function $\eta\in W^{1,\infty}_0(\B_1)$ satisfying \eqref{definition of eta}, such that
\begin{equation}\label{second estimate lemma}
	\int\limits_{\B_1} \reallybigtilde{A\circ b^{-1}}(|u\nabla \eta|)\dex\le c\frac{\gamma^m}{(\sigma-\rho)^{m-1+\frac{m}{p}}\rho^{(n-1)(\frac{m}{p}-1)}}\mathcal{F}_{\rho,\sigma}(u)^{\frac{m}{p}}
\end{equation}
for some positive constant $c=c(n,m,p,L)$.\\
In $(i)$ and $(ii)$ above, $\gamma$ is the constant provided by parts $(i)$ and $(ii)$ of Theorem \ref{Orlicz-Sobolev on spheres}, respectively.
\end{lemma}
Finally, we recall the following classical result --- see Lemma $6.1$ in \cite{Giusti}.
\begin{lemma}\label{algebraic lemma}
Let $Z\colon[\rho,\sigma]\to[0,\infty)$ be a bounded function. Assume that for $\rho\le r_1\le r_2\le \sigma$ we have
$$Z(r_1)\le \theta Z(r_2) +a(r_2-r_1)^{-\beta}+b$$
for some $\theta\in[0,1)$, $\beta>0$ $a,b\ge0$. Then
$$Z(\sigma)\le c(\sigma-\rho)^{-\beta}a+cb$$
for some constant $c=c(\theta,\beta)>1$.
\end{lemma}

\section{Proof of Theorem \ref{boundedness theorem}}\label{Proof of the Theorem}
Without loss of generality, we assume that all the properties listed in Lemma \ref{redefinition of functions and parameters} hold.

Let $u\in W^1L^B_{\mathrm{loc}}(\Omega)$ be a weak solution to Equation \eqref{equation}. By scaling and translation arguments, it suffices to assume that $\mathbb{B}_1\Subset\Omega$ and to prove that $u$ is locally bounded in $\mathbb{B}_{\sfrac{1}{2}}$. Hence, from this point on we assume that $\mathbb{B}_1\Subset\Omega$. We also fix a positive constant $\gamma_n$, depending only on $n$, such that \eqref{first case orlicz-sobolev} holds for every $\rho\in[\frac{1}{2},1]$.
\subsection{Caccioppoli-type estimates}
For every $k\ge0$ and $0<\rho<1$ we set
$$\Omega_{k,\rho}\Def\{x\in\Omega : |u(x)|> k\}\cap \B_\rho$$
and
$$M_{k,\rho}\Def\int\limits_{\B_\rho}A( (|u|-k)_+)+A(|\nabla(|u|-k)_+  |)\dex.$$
We also define the measurable function
$$G_k(u)\Def(|u|-k)_+\mathrm{sign}(u).$$
As one can verify, $G_k(u)\in W^1 L^B_{\mathrm{loc}}(\Omega)$ and $\nabla G_k(u)=\nabla u \chi_{\{|u|>k\}}$.

Let $p$ and $m$ be as in the statement of Lemma \eqref{redefinition of functions and parameters}. For notation convenience we denote by $p^\star$ the number $\frac{np}{n-p}$.

To begin, we prove the following lemma.
\begin{lemma}\label{caccioppoli}
$(1)$ Suppose that $B$ has super-linear growth. Then:

$(i)$ If \eqref{first case main thm} holds and \eqref{main condition} is fulfilled, there exists a positive constant $c=c(A,b,n,m,L,\lambda,R)$ such that
\begin{multline}\label{1 caccioppoli}
	\int\limits_{ \Omega _{k,\rho}}A(|\nabla u|)\dex\le \int\limits_{ \Omega _{k,\sigma}}E(|u|)\dex+ c\bigg(\frac{ \Phi_m(M_{k,\sigma}^{\frac{1}{n-1}})}{(\sigma-\rho)^{m+\frac{m}{n-1}}}M_{k,\sigma}+|\Omega _{k,\sigma}|^{\frac{1}{r'}}\|A\circ b^{-1}(\lambda\phi)\|_{L^r(\Omega _{k,\sigma})}+\\+\|\psi\|_{L^{ s }(\Omega_{k,\sigma})} M_{k,\sigma}^{\frac{1}{p}}|\Omega_{k,\sigma}|^{1-\frac{1}{s}-\frac{1}{p^\star}}+|\Omega_{k,\sigma}|\bigg)
\end{multline}
for all $k\ge0$ and all $\frac{1}{2}\le \rho<\sigma<1$.

$(ii)$ If \eqref{second case main thm} holds, then there exists a positive constant $c=c(A,b,n,m,p,L,\lambda,R)$ such that
\begin{multline}\label{2 caccioppoli}
	\int\limits_{ \Omega _{k,\rho}}A(|\nabla u|)\dex\le \int\limits_{ \Omega _{k,\sigma}}E(|u|)\dex+  c \bigg(\frac{1}{(\sigma-\rho)^{m+\frac{m}{p}}}M_{k,\sigma}^{\frac{m}{p}}+|\Omega _{k,\sigma}|^{\frac{1}{r'}}\|A\circ b^{-1}(\lambda\phi)\|_{L^r(\Omega _{k,\sigma})}+\\+\|\psi\|_{L^{ s }(\Omega_{k,\sigma})} M_{k,\sigma}^{\frac{1}{p}}|\Omega_{k,\sigma}|^{1-\frac{1}{ s }-\frac{1}{p^\star}}+|\Omega _{k,\sigma}|\bigg)
\end{multline}
for all $k\ge0$ and all $\frac{1}{2}\le \rho<\sigma<1$.

$(2)$ Suppose that $B$ has linear growth. Then, there exists a positive constant $c=c(A,n,L,R)$ such that
\begin{multline}\label{0 caccioppoli}
	\int\limits_{\Omega _{k,\rho}}A(|\nabla u|)\dex\le \int\limits_{ \Omega _{k,\sigma}}E(|u|)\dex+ c\bigg(\big(1+\|\phi\|_{L^\infty{(\Omega _{k,\sigma})}}\big)\frac{1}{\sigma-\rho}|\Omega _{k,\sigma}|^{\frac{1}{n}}M_{k,\sigma}+\\+\|\psi\|_{L^{s}(\Omega_{k,\sigma})}|\Omega_{k,\sigma}|^{\frac{1}{n}-\frac{1}{s}}M_{k,\sigma} +|\Omega_{k,\sigma}|\bigg)
\end{multline}
for all $k\ge0$ and all $\frac{1}{2}\le \rho<\sigma<1$.
\end{lemma}
\begin{proof} We divide the proof into two parts, according to the growth rate of $B$.

\emph{Part $1$: $B$ has super-linear growth.} Here, \eqref{assumption 2} holds. Relying on Lemma \eqref{lemma construction of cut off}, the argument varies slightly depending on whether \eqref{first case main thm} or \eqref{second case main thm} applies.

$(i)$ First assume that \eqref{first case main thm} is in effect, and \eqref{main condition} is satisfied. Fix $k\ge0$ and radii satisfying $\frac{1}{2}\le\rho<\sigma<1$. We divide the proof into steps.

\emph{Step $1$.} Let $\eta$ be the cut-off function provided by part $(i)$ of Lemma \eqref{lemma construction of cut off} applied to $ (|u|-k)_+  $. Choosing $v=\eta^m  G_k(u) $ as a test function in \eqref{weak formulation} yields
\begin{equation*}
	\int\limits_{ \Omega _{k,\sigma}} \eta^m\A(x,\nabla u)\cdot \nabla u\dex=\int\limits_{ \Omega _{k,\sigma}\setminus \B_\rho} -m\eta^{m-1} G_k(u) \A(x,\nabla u)\cdot \nabla \eta  \dex +\int\limits_{ \Omega _{k,\sigma}} f(x,u)\eta^m  G_k(u)\dex.
\end{equation*}
Applying \eqref{ellipticity} on the left-hand side and \eqref{growth} on the right-hand side we infer
\begin{align}
	\int\limits_{ \Omega _{k,\rho}\cap\{|\nabla u|\ge R\}} A(|\nabla u|)&\dex\le m\int\limits_{ \Omega _{k,\sigma}\setminus \B_\rho} \eta^{m-1} (|u|-k)_+  |\A(x,\nabla u)\cdot \nabla \eta|\dex+\int\limits_{ \Omega _{k,\sigma}}|f(x,u)| (|u|-k)_+  \dex\nonumber\\
	&\le m\int\limits_{ \Omega _{k,\sigma}\setminus \B_\rho} b(|\nabla u|)  (|u|-k)_+   |\nabla \eta|+\phi (|u|-k)_+   |\nabla \eta|\dex+\int\limits_{ \Omega _{k,\sigma}}|f(x,u)| (|u|-k)_+  \dex.\label{first case caccioppoli}
\end{align}
We want to estimate the integrals on the right-hand side of \eqref{first case caccioppoli}.

\emph{Step $2$.} We estimate the first integral in \eqref{first case caccioppoli}. Note first that the following inequalities hold:
\begin{align}
	\int\limits_{ \Omega _{k,\sigma}\setminus \B_\rho}  (|u|-k)_+  b(|\nabla u|)|\nabla \eta|\dex &\le \int\limits_{ \Omega _{k,\sigma}\setminus \B_\rho}A(|\nabla u|)\dex+\int\limits_{ \Omega _{k,\sigma}\setminus \B_\rho}\reallybigtilde{A\circ b^{-1}}( (|u|-k)_+  |\nabla \eta|)\dex\nonumber\\
	&\le\int\limits_{ \Omega _{k,\sigma}\setminus \B_\rho}A(|\nabla u|)\dex+c \frac{ \Phi_ m(\gamma \mathcal{F}_{\rho,\sigma}( (|u|-k)_+  )^{\frac{1}{n-1}})}{\rho^m(\sigma-\rho)^{n'm}}\mathcal{F}_{\rho,\sigma}( (|u|-k)_+  )\label{application lemma for caccioppoli 1}
\end{align}
for some constant $c=c(n,m,L)\ge1$, where the first inequality follows by applying Young's inequality to $A\circ b^{-1}$, whereas the second one is a consequence of \eqref{first estimate lemma}. Since $\rho\ge\frac{1}{2}$ and $\mathcal{F}_{\rho,\sigma}( (|u|-k)_+  )\le M_{k,\sigma}$, \eqref{application lemma for caccioppoli 1} produces
\begin{equation}\label{estimate first term first integral caccioppoli 1}
	\int\limits_{ \Omega _{k,\sigma}\setminus \B_\rho}  (|u|-k)_+  b(|\nabla u|)|\nabla \eta|\dex\le \int\limits_{ \Omega _{k,\sigma}\setminus \B_\rho}A(|\nabla u|)\dex+c \frac{ \Phi_ m(\gamma M_{k,\sigma}^{\frac{1}{n-1}})}{2^m(\sigma-\rho)^{n'm}}M_{k,\sigma}.
\end{equation}
In order to estimate
$$\int\limits_{ \Omega _{k,\sigma}\setminus \B_\rho}\phi (|u|-k)_+ |\nabla \eta|\dex,$$
recall that, by \eqref{summability inequality redefinition} and \eqref{requirements on phi redefinition}, $A\circ b^{-1}(\lambda\phi)\in L^r_{\mathrm{loc}}(\Omega)$, with $r>\frac{n}{p}$. Then, the following chain of inequalities holds:
\begin{align}
	\int\limits_{ \Omega _{k,\sigma}\setminus \B_\rho}\phi (|u|-k)_+   |\nabla \eta|\dex &\le \int\limits_{ \Omega _{k,\sigma}}A\circ b^{-1}(\lambda\phi)\dex+\int\limits_{ \Omega _{k,\sigma}}\reallybigtilde{A\circ b^{-1}}\Big( \frac{(|u|-k)_+ |\nabla \eta|}{\lambda}\Big)\dex\nonumber\\
	&\le |\Omega _{k,\sigma}|^{1-\frac{1}{r}}\|A\circ b^{-1}(\lambda\phi)\|_{L^r(\Omega _{k,\sigma})}+c\int\limits_{ \Omega _{k,\sigma}}\reallybigtilde{A\circ b^{-1}}\big( {(|u|-k)_+|\nabla \eta|}\big)\dex\nonumber\\
	&\le |\Omega _{k,\sigma}|^{1-\frac{1}{r}}\|A\circ b^{-1}(\lambda\phi)\|_{L^r(\Omega _{k,\sigma})}+c\frac{ \Phi_ m(\gamma M_{k,\sigma}^{\frac{1}{n-1}})}{(\sigma-\rho)^{n'm}}M_{k,\sigma}
	\label{estimate second term first integral caccioppoli 1}
\end{align}
for some $c=c(n,m,L,\lambda)\ge1$, where in the first inequality we have applied Young's inequality, in the second one H\"older's inequality and the $\Delta_2$-condition on $\reallybigtilde{A\circ b^{-1}}$, and in the third one the estimate in \eqref{first estimate lemma}, and the inequalities $\rho\ge\frac{1}{2}$ and $\mathcal{F}_{\rho,\sigma}( (|u|-k)_+  )\le M_{k,\sigma}$.\\
The estimates \eqref{estimate first term first integral caccioppoli 1} and \eqref{estimate second term first integral caccioppoli 1} yield
\begin{multline}\label{first estimate first case caccioppoli}
	\int\limits_{\Omega _{k,\sigma}\setminus \B_\rho} b(|\nabla u|) (|u|-k)_+ |\nabla \eta|+\phi (|u|-k)_+   |\nabla \eta|\dex\le\int\limits_{ \Omega _{k,\sigma}\setminus \B_\rho}A(|\nabla u|)\dex+\\
	+c\bigg( \frac{ \Phi_ m(\gamma M_{k,\sigma}^{\frac{1}{n-1}})}{(\sigma-\rho)^{n'm}}M_{k,\sigma}+|\Omega _{k,\sigma}|^{\frac{1}{r'}}\|A\circ b^{-1}(\lambda\phi)\|_{L^r(\Omega _{k,\sigma})}\bigg)
\end{multline}
for some $c=c(A,b,n,m,L,\lambda)\ge1$.

\emph{Step $3$.} Let us consider the second integral in \eqref{first case caccioppoli}. Taking into account \eqref{summability inequality redefinition}, the assumption \eqref{growth on f} together with H\"older's inequality implies
\begin{align*}
	\int\limits_{\Omega _{k,\sigma}}|f(x,u)|\eta^m  (|u|-k)_+  \dex&\le \int\limits_{ \Omega _{k,\sigma}}\psi\eta^m  (|u|-k)_+  \dex+\int\limits_{ \Omega _{k,\sigma}}\frac{E(|u|)}{|u|}\eta^m  (|u|-k)_+  \dex\nonumber\\
	&\le\|\psi\|_{L^{ s }(\Omega_{k,\sigma})}\| (|u|-k)_+  \|_{L^{p^\star}(\Omega _{k,\sigma})}|\Omega_{k,\sigma}|^{1-\frac{1}{ s }-\frac{1}{p^\star}}+\int\limits_{\Omega_{k,\sigma}}E(|u|)\dex.
\end{align*}
By \ref{sobolev conjugate dominates the conjugate of the lower exponent} and part $(i)$ of Theorem \ref{Orlicz-Sobolev embedding}, we have
\begin{equation}\label{estimate on pstar norm of u}
	\| (|u|-k)_+  \|_{L^{p^\star}(\Omega_{k,\sigma})}\le L\gamma_n M_{k,\sigma}^{\frac{1}{n}}\bigg( \int\limits_{\Omega_{k,\sigma}} A_n\bigg(\frac{ (|u|-k)_+  }{\gamma_n M_{k,\sigma}^\frac{1}{n}}\bigg)\dex\bigg)^\frac{1}{p^\star}\le L\gamma_n M_{k,\sigma}^{\frac{1}{n}+\frac{1}{p^\star}}=L\gamma_n M_{k,\sigma}^{\frac{1}{p}}.
\end{equation}
Therefore
\begin{equation}
	\int\limits_{ \Omega _{k,\sigma}}|f(x,u)|\eta^m  (|u|-k)_+  \dex\le L\gamma_n \|\psi\|_{L^{ s }(\Omega_{k,\sigma})} M_{k,\sigma}^{\frac{1}{p}}|\Omega_{k,\sigma}|^{1-\frac{1}{ s }-\frac{1}{p^\star}}+\int\limits_{ \Omega _{k,\sigma}}E(|u|)\dex.\label{second estimate first case caccioppoli}
\end{equation}

\emph{Step $4$.} Putting together \eqref{first case caccioppoli}, \eqref{first estimate first case caccioppoli}, and \eqref{second estimate first case caccioppoli} entails
\begin{multline*}
	\int\limits_{ \Omega _{k,\rho}\cap\{|\nabla u|\ge R\}} A(|\nabla u|)\dex\le m\int\limits_{ \Omega _{k,\sigma}\setminus \B_\rho}A(|\nabla u|)\dex+\int\limits_{ \Omega _{k,\sigma}}E(|u|)\dex+ \\+c\bigg(\frac{ \Phi_ m(M_{k,\sigma}^{\frac{1}{n-1}})}{(\sigma-\rho)^{n'm}}M_{k,\sigma}+|\Omega _{k,\sigma}|^{\frac{1}{r'}}\|A\circ b^{-1}(\lambda\phi)\|_{L^r(\Omega _{k,\sigma})}+\\+\|\psi\|_{L^{ s }(\Omega_{k,\sigma})} M_{k,\sigma}^{\frac{1}{p}}|\Omega_{k,\sigma}|^{1-\frac{1}{ s }-\frac{1}{p^\star}}\bigg)
\end{multline*}
for some $c=c(A,b,n,m,L,\lambda)\ge1$.
Since $\Omega_{k,\rho}$ splits into $\Omega_{k,\rho}\cap\{|\nabla u|\ge R\}$ and $ \Omega _{k,\rho}\cap\{|\nabla u|<R\}$, the previous formula yields
\begin{multline}\label{almost final caccioppoli}
	\int\limits_{ \Omega _{k,\rho}} A(|\nabla u|)\dex\le A(R) | \Omega _{k,\rho}|+ m\int\limits_{ \Omega _{k,\sigma}\setminus \B_\rho}A(|\nabla u|)\dex+\int\limits_{ \Omega _{k,\sigma}}E(|u|)\dex+\\+c\bigg(\frac{ \Phi_ m(M_{k,\sigma}^{\frac{1}{n-1}})}{(\sigma-\rho)^{n'm}}M_{k,\sigma}+|\Omega _{k,\sigma}|^{\frac{1}{r'}}\|A\circ b^{-1}(\lambda\phi)\|_{L^r(\Omega _{k,\sigma})}+\\+\|\psi\|_{L^{ s }(\Omega_{k,\sigma})} M_{k,\sigma}^{\frac{1}{p}}|\Omega_{k,\sigma}|^{1-\frac{1}{ s }-\frac{1}{p^\star}}\bigg).
\end{multline}
We now use the ``hole filling'' method, namely, we add
$$m\int\limits_{ \Omega _{k,\rho}} A(|\nabla u|)\dex$$
to both sides of \eqref{almost final caccioppoli}, and divide by $m+1$. We obtain
\begin{multline*}
	\int\limits_{ \Omega _{k,\rho}} A(|\nabla u|)\dex\le \frac{m}{m+1}\int\limits_{ \Omega _{k,\sigma}}A(|\nabla u|)\dex+\int\limits_{ \Omega _{k,\sigma}}E(|u|)\dex+\\+c\bigg(\frac{ \Phi_ m(M_{k,\sigma}^{\frac{1}{n-1}})}{(\sigma-\rho)^{n'm}}M_{k,\sigma}+|\Omega _{k,\sigma}|^{\frac{1}{r'}}\|A\circ b^{-1}(\lambda\phi)\|_{L^r(\Omega _{k,\sigma})}+\\+\|\psi\|_{L^{ s }(\Omega_{k,\sigma})} M_{k,\sigma}^{\frac{1}{p}}|\Omega_{k,\sigma}|^{1-\frac{1}{ s }-\frac{1}{p^\star}}+|\Omega _{k,\sigma}|\bigg),
\end{multline*}
for some constant $c=c(A,b,n,m,L,\lambda,R)\ge1$. Hence, the estimate \eqref{1 caccioppoli} follows from Lemma \ref{algebraic lemma}.

$(ii)$ Assume that \eqref{second case main thm} is in force. To prove \eqref{2 caccioppoli}, one can adapt the argument in part $(i)$. The main difference consists in using the estimate \eqref{second estimate lemma} on the right-hand side of \eqref{first case caccioppoli} to deduce
\begin{multline*}
	\int\limits_{ \Omega _{k,\sigma}\setminus \B_\rho} b(|\nabla u|)  (|u|-k)_+   |\nabla \eta|+\phi (|u|-k)_+   |\nabla \eta|\dex\le\int\limits_{ \Omega _{k,\sigma}\setminus \B_\rho}A(|\nabla u|)\dex+\\
	+c\bigg( \frac{\gamma^m}{(\sigma-\rho)^{m+\frac{m}{p}}}M_{k,\sigma}^{\frac{m}{p}}+|\Omega _{k,\sigma}|^{\frac{1}{r'}}\|A\circ b^{-1}(\lambda\phi)\|_{L^r(\Omega _{k,\sigma})}\bigg)
\end{multline*}
for some $c=c(A,b,n,m,L)\ge1$.

\emph{Part $2$: $B$ has linear growth.} In this setting, no absorption argument is needed. Fix a level $k\ge0$ and radii $\frac{1}{2}\le\rho<\sigma<1$. Let $\eta\in C^1_\mathrm{c}(\B_1)$ be any cut-off function corresponding to these radii, namely
	\begin{equation*}
		0\le\eta\le1,\quad \eta=1\ \text{in}\ \B_\rho,\quad \eta=0\ \text{in}\ \B_1\setminus\B_\sigma,\quad \|\nabla \eta\|_{L^\infty(\B_1)}\le \frac{2}{\sigma-\rho}.
	\end{equation*}
	Choosing $v=\eta G_k(u) $ as a test function in \eqref{weak formulation}, and exploiting \eqref{ellipticity} and \eqref{growth} as in the super-linear case, we deduce
	\begin{align}\label{inequality zero linear case caccioppoli}
		\int\limits_{ \Omega _{k,\rho}\cap\{|\nabla u|\ge R\}} A(|\nabla u|) \le \int\limits_{ \Omega _{k,\sigma}} (b(|\nabla u|)+\phi) |\nabla \eta|(|u|-k)_+\dex+\int\limits_{ \Omega _{k,\sigma}}|f(x,u)| (|u|-k)_+  \dex= I_1+I_2.
	\end{align}
	It suffices to estimate $I_1$. Since $B$ has linear growth, the assumption \eqref{requirements on phi redefinition} can be rewritten as
	$$\phi\in L^\infty_{\mathrm{loc}}(\Omega).$$
	Then, we have
\begin{align*}
I_1&\le(L+\|\phi\|_{L^\infty{(\Omega _{k,\sigma})}})\frac{2}{\sigma-\rho}\int\limits_{\Omega _{k,\sigma}} (|u|-k)_+   \dex\\
&\le (L+\|\phi\|_{L^\infty{(\Omega _{k,\sigma})}})\frac{2}{\sigma-\rho}|\Omega _{k,\sigma}|^{\frac{1}{n}}\|(|u|-k)_+\|_{L^{\frac{n}{n-1}}(\Omega _{k,\sigma})}\\
&\le (L+\|\phi\|_{L^\infty{(\Omega _{k,\sigma})}})\frac{2L\gamma_n}{\sigma-\rho}|\Omega _{k,\sigma}|^{\frac{1}{n}}M_{k,\sigma},
\end{align*}
where the first inequality follows from the bound in \eqref{b is bounded}, the second from H\"older's inequality, and the third from the last estimate in \eqref{estimate on pstar norm of u}. The conclusion follows by treating $I_2$ and the leftmost side of \eqref{inequality zero linear case caccioppoli} as in part $(i)$ of the proof of the super-linear case.
\end{proof}
\subsection{Recursive formula} Note that, for every $\sigma\in[\frac{1}{2},1]$, $k\mapsto M_{k,\sigma}$ is decreasing in $[0,\infty)$. We set
$$\widehat{c}\Def\max\{\gamma_n,1\}.$$
Noting that, by \eqref{summability inequality redefinition},
\begin{equation}\label{positive terms for last computations}
\frac{sp -n}{ s (n-p)}>0 \quad \text{and}\quad \frac{r'p-(r'-1)n}{r'(n-p)}>0,
\end{equation}
we now prove the following one step improvement.
\begin{lemma}\label{lemma one step improvement}
Let $\frac{1}{2}\le\rho<\sigma<1$. Suppose that there exists $h\ge0$ such that
\begin{equation}\label{first term of sequence small}
	\widehat{c}L M_{h,\sigma}\le 1.
\end{equation}
Let $\alpha=\alpha(n,p,r,s)$ be the positive constant
$$\alpha=\min\Big\{\frac{1}{n},\frac{sp-n}{s(n-p)},\frac{r'p-(r'-1)n}{r'(n-p)}\Big\}.$$
Then:

$(1)$ If $B$ has super-linear growth, there exists a positive constant $c=c(A,b,n,m,p,r,L,\lambda,R)$ such that, for every $k>h$,
\begin{equation}\label{one step improvement}
	M_{k,\rho}\le c\bigg(1+L^{\log_2(\frac{k}{k-h})}+\frac{1}{(\sigma-\rho)^{\beta}}+\frac{\|A\circ b^{-1}(\lambda\phi)\|_{L^r(\Omega _{k,\sigma})}}{(k-h)^{\sfrac{p^\star}{r'}}}+\frac{\|\psi\|_{L^{s}(\Omega_{k,\sigma})}}{(k-h)^{p^\star-1-\sfrac{p^\star}{s}}}+\frac{1}{(k-h)^{p^\star}}\bigg)M_{h,\sigma}^{1+\alpha},
\end{equation}
where $\beta=\beta(n,p,m)$ is given by
$$\beta=m+\max\Big\{\frac{m}{n-1},\frac{m}{p}\Big\}.$$

$(2)$ If $B$ has linear growth, there exists a positive constant $c=c(A,b,n,L,R)$ such that, for every $k>h$,
\begin{equation}\label{one step improvement bis}
M_{k,\rho}\le c\bigg(1+L^{\log_2(\frac{k}{k-h})}+\frac{1+\|\phi\|_{L^\infty{(\Omega _{k,\sigma})}}}{(\sigma-\rho)(k-h)^{\frac{1}{n-1}}}+\frac{\|\psi\|_{L^{s}(\Omega_{k,\sigma})}}{(k-h)^{\frac{s-n}{s(n-1)}}}+\frac{1}{(k-h)^{\frac{n}{n-1}}}\bigg)M_{h,\sigma}^{1+\alpha}.
\end{equation}
\end{lemma}
\begin{proof}
We derive \eqref{one step improvement} and \eqref{one step improvement bis} from Lemma \eqref{caccioppoli}. Let $k>h\ge0$ and note that, by \eqref{first term of sequence small}, $M_{k,\sigma}\le M_{h,\sigma}\le1$.

\emph{Step $1$.} We first estimate $|\Omega_{k,\sigma}|$. Observe that
\begin{multline*}
	| \Omega _{k,\sigma}|(k-h)^{p^\star}\le \int\limits_{ \Omega _{h,\sigma}} (|u|-h)^{p^\star}\dex= \gamma_n^{p^\star} M_{h,\sigma}^{\frac{p^\star}{n}}\int\limits_{ \Omega _{h,\sigma}} \bigg(\frac{(|u|-h)_+}{\gamma_n M_{h,\sigma}^{\frac{1}{n}}}\bigg)^{p^\star}\dex\\
	\le c\gamma_n^{p^\star} M_{h,\sigma}^{\frac{{p^\star}}{n}}\int\limits_{ \Omega _{h,\sigma}} A_n\bigg(\frac{|u|-h}{\gamma_n M_{h,\sigma}^{\frac{1}{n}}}\bigg)\dex\le c\gamma_n^{p^\star} M_{h,\sigma}^{1+\frac{p}{n-p}},
\end{multline*}
for some $c=c(n,R,p,A)>0$, where the second inequality follows from \ref{sobolev conjugate dominates the conjugate of the lower exponent}, while the third one follows from part $(i)$ of Theorem \eqref{Orlicz-Sobolev embedding}. We deduce that
\begin{equation}\label{estimate on measure}
	| \Omega _{k,\sigma}|\le c\frac{1}{(k-h)^{p^\star}}M_{h,\sigma}^{1+\frac{p}{n-p}}
\end{equation}
with $c=c(n,R,p,A)$.

\emph{Step $2$.} Let us estimate
$$\int\limits_{\Omega _{k,\sigma}}E(|u|)\dex.$$
Since $E$ is convex and it satisfies the $\Delta_2$-condition, for \Ae $x\in\Omega_{k,\sigma}$ it holds
$$E(|u(x)|)=E(|u(x)|-k+k)\le\frac{1}{2}\big(E\big(2(|u(x)|-k)\big)+E(2k)\big)\le \frac{L}{2}\big(E(|u(x)|-k)+E(k)\big).$$
From \eqref{growth inequalities for E} and the assumption \eqref{first term of sequence small} one infers that
\begin{align}
	\int\limits_{\Omega _{k,\sigma}}E(|u|-k)\dex\le \int\limits_{\Omega _{h,\sigma}}E(|u|-h)\dex&\le\int\limits_{\Omega _{h,\sigma}}A_n(L(|u|-h))\dex\nonumber\\
	& \le \gamma_n LM_{h,\sigma}^{\frac{1}{n}}\int\limits_{\Omega _{h,\sigma}}A_n\bigg(\frac{|u|-h}{\gamma_n M_{h,\sigma}^{\frac{1}{n}}}\bigg)\dex\le \gamma_n LM_{h,\sigma}^{1+\frac{1}{n}}\label{first inequality for estimating integral of E}
\end{align}
where in the third and fourth inequalities we have exploited the convexity of $A_n$ and part $(i)$ in Theorem \ref{Orlicz-Sobolev embedding}, respectively. On the other hand, by \eqref{growth inequalities for E} one has
\begin{align}
	\int\limits_{\Omega _{k,\sigma}}E(k)\dex&=E\bigg(\frac{k}{k-h}(k-h)\bigg)|\Omega _{k,\sigma}|\nonumber\\
	&\le L^{\lceil\log_2(\frac{k}{k-h})\rceil} E((k-h))|\Omega _{k,\sigma}|\nonumber\\
	&\le L^{\lceil\log_2(\frac{k}{k-h})\rceil}\int\limits_{\Omega _{k,\sigma}} E(|u|-h)\dex,\label{second inequality for estimating integral of E}
\end{align}
where $\lceil\cdot\rceil$ denotes the ceiling function. Then, the chains \eqref{first inequality for estimating integral of E} and \eqref{second inequality for estimating integral of E} produce the following estimate:
\begin{equation}\label{estimate on integral}
	\int\limits_{\Omega _{k,\sigma}}E(|u|)\dex\le \gamma_n\frac{L^2}{2}\big(1+L^{\lceil\log_2(\frac{k}{k-h})\rceil}\big)M_{h,\sigma}^{1+\frac{1}{n}}.
\end{equation}

\emph{Step $3$.} We now estimate
$$\int\limits_{\Omega _{k,\rho}}A(|u|-k)\dex.$$
Since $A\lesssim A_n$ near infinity, there exists a positive constant $c=c(n,A)$ such that
\begin{align*}
	\int\limits_{\Omega _{k,\rho}}A(|u|-k)\dex&\le \int\limits_{\Omega _{k,\rho}}A_n(|u|-k)\dex+c|\Omega _{k,\rho}|\\
	&\le \int\limits_{\Omega _{h,\sigma}}A_n(|u|-h)\dex+c|\Omega _{k,\sigma}|.
\end{align*}
As a consequence, \eqref{estimate on measure} entails
\begin{equation}\label{estimate on integral of A(|u|-k)}
	\int\limits_{\Omega _{k,\rho}}A( |u|-k )\dex\le \gamma_n M_{k,\sigma}^{1+\frac{1}{n}}+c\frac{1}{(k-h)^{p^\star}}M_{h,\sigma}^{1+\frac{p}{n-p}}\le \bigg(\gamma_n+c\frac{1}{(k-h)^{p^\star}}\bigg)M_{k,\sigma}^{1+\frac{1}{n}}
\end{equation}
for some $c=c(n,R,p,A)>0$.

\emph{Step $4$: $B$ has super-linear growth.} Assume that \eqref{first case main thm} is in effect. Adding
\begin{equation}\label{mass of u}
	\int_{\Omega_{k,\rho}} A(|u|-k),dx
\end{equation}
to both sides of \eqref{1 caccioppoli}, and estimating the right–hand side of this new inequality by means of \eqref{estimate on integral} and \eqref{estimate on integral of A(|u|-k)}, we obtain
\begin{multline*}
	M_{k,\rho}\le c\bigg(1+L^{\lceil\log_2(\frac{k}{k-h})\rceil}+\frac{1}{(k-h)^{p^\star}}\bigg)M_{k,\sigma}^{1+\frac{1}{n}}\\+ c\bigg(\frac{ \Phi_m(M_{k,\sigma}^{\frac{1}{n-1}})}{(\sigma-\rho)^{m+\frac{m}{n-1}}}M_{k,\sigma}+|\Omega _{k,\sigma}|^{\frac{1}{r'}}\|A\circ b^{-1}(\lambda\phi)\|_{L^r(\Omega _{k,\sigma})}+\\+\|\psi\|_{L^{ s }(\Omega_{k,\sigma})} M_{k,\sigma}^{\frac{1}{p}}|\Omega_{k,\sigma}|^{1-\frac{1}{s}-\frac{1}{p^\star}}+|\Omega_{k,\sigma}|\bigg)
\end{multline*}
for some $c=c(A,b,n,m,L,\lambda,R,p)>0$. Thus, to complete the proof, it suffices to estimate the second parenthesis above. Since $M_{k,\sigma}\le1$, we have
\begin{equation}\label{estimate on Phi}
	\Phi_m(M_{k,\sigma}^{\frac{1}{n-1}})M_{k,\sigma}=M_{k,\sigma}^{1+\frac{1}{n-1}}\le M_{h,\sigma}^{1+\frac{1}{n-1}}. 
\end{equation}
Furthermore, since $p^\star=\frac{np}{n-p}$ and $r'=\frac{r}{r-1}$, straightforward computations give
\begin{equation}\label{almost last computation of the proof}
	\frac{n}{n-p}\Big(1-\frac{1}{s}-\frac{1}{p^\star}\Big)+\frac{1}{p}
	=1+\frac{sp-n}{s(n-p)},
\end{equation}
and
\begin{equation}\label{last computation of the proof}
	\frac{n}{n-p}\frac{1}{r'}
	=1+\frac{r'p-(r'-1)n}{r'(n-p)},
\end{equation}
where both $\frac{sp -n}{s(n-p)}$ and $\frac{r'p-(r'-1)n}{r'(n-p)}$ are positive by \eqref{positive terms for last computations}. Thus, by \eqref{estimate on measure}, there exist constants $c_1=c_1(n,R,A,p,s)>0$ and $c_2=c_2(n,R,A,r)>0$ such that
\begin{equation*}
	|\Omega_{k,\sigma}|^{1-\frac{1}{ s }-\frac{1}{p^\star}}\le c_1\frac{1}{(k-h)^{p^\star(1-\frac{1}{ s }-\frac{1}{p^\star})}}M_{h,\sigma}^{1+\frac{sp -n}{s(n-p)}}\le c_1\frac{1}{(k-h)^{p^\star(1-\frac{1}{ s }-\frac{1}{p^\star})}}M_{h,\sigma}^{1+\alpha}
\end{equation*}
and
\begin{equation*}
	|\Omega_{k,\sigma}|^{\frac{1}{r'}}\le c_2\frac{1}{(k-h)^{\sfrac{p^\star}{r'}}}M_{h,\sigma}^{1+\frac{r'p-(r'-1)n}{r'(n-p)}}\le c_2\frac{1}{(k-h)^{\sfrac{p^\star}{r'}}}M_{h,\sigma}^{1+\alpha}.
\end{equation*}
Hence, inserting the above bounds into the energy estimate yields \eqref{one step improvement}, completing the argument.

If \eqref{second case main thm} is in force, the argument proceeds exactly as in the previous case, with the only difference that the estimate \eqref{estimate on Phi} is replaced by
$$M_{k,\sigma}^{\frac{m}{p}}\le M_{h,\sigma}^{1+\frac{1}{n-1}},$$
which follows from $m>pn'$ together with $M_{h,\sigma}\le1$.

\emph{Step $5$: $B$ has linear growth.} In this setting, we exploit the energy estimate in \eqref{0 caccioppoli}. Adding \eqref{mass of u} to both sides of the latter, \eqref{one step improvement bis} follows easily estimating the right-hand side via \eqref{estimate on measure}, \eqref{estimate on integral}, and \eqref{estimate on integral of A(|u|-k)}.
\end{proof}
Let $K>0$ to be chosen. For every $j\in\N$, set
$$h_j=K\bigg(1-\frac{1}{2^{j+1}}\bigg),\quad \sigma_j=\frac{1}{2}+\frac{1}{2^{j+2}},$$
and
$$M_j\Def M_{h_j,\sigma_j}.$$
Note that, thanks to the definitions of $\sigma_j$ and $h_j$, the sequence $(M_j)$ is decreasing.

If $\widehat{c}L M_j^{\frac{1}{n}}\le 1$ for some $j$, choosing $h=h_j$, $k=h_{j+1}$, $\sigma=\sigma_{j}$ , and $\rho=\sigma_{j+1}$ in Lemma \ref{lemma one step improvement}, the estimates \eqref{one step improvement} and \eqref{one step improvement bis} rewrite as
\begin{multline*}
	M_{j+1}\le c\bigg(1+L^{j+2}+2^{(j+3)\beta}+\|A\circ b^{-1}(\lambda\phi)\|_{L^r(\Omega_{K,\sfrac{3}{4}})}\frac{2^{(j+2)\sfrac{p^\star}{r'}}}{K^{\sfrac{p^\star}{r'}}}+                 \\
	+\|\psi\|_{L^{s}(\Omega_{K,\sfrac{3}{4}})}\frac{2^{(j+2)(p^\star-1-\sfrac{p^\star}{ s })}}{K^{p^\star-1-\sfrac{p^\star}{s}}}+\frac{2^{(j+2)p^\star}}{K^{p^\star}}\bigg)M_j^{1+\alpha}
\end{multline*}
and
$$\begin{aligned}
M_{j+1} \le c\bigg(1 + L^{j+2} + \frac{2^{\frac{2j+5}{n-1}}}{K^{\frac{1}{n-1}}}
\big(1+\|\phi\|_{L^\infty(\Omega_{K,\sfrac{3}{4}})}\big) + \|\psi\|_{L^{s}(\Omega_{K,\sfrac{3}{4}})}
\frac{2^{(j+2)\left(\frac{s-n}{s(n-1)}\right)}}{K^{\frac{s-n}{s(n-1)}}} + \frac{2^{(j+2)n'}}{K^{n'}}\bigg)
M_j^{1+\alpha}
\end{aligned}$$
for some $c=c(A,b,n,m,p,r,L,\lambda,R)>0$. This, in turn, implies
$$M_{j+1}\le c_K \tau^{j}M_j^{1+\alpha}$$
where $\tau=\max\{L,2^{\frac{2}{n-1}},2^{\beta},2^{p^\star}\}$ and
\begin{multline}\label{definition of cK}
c_K=c\max\bigg\{1+L^2+2^{3\beta}+\|A\circ b^{-1}(\lambda\phi)\|_{L^r(\Omega_{K,\sfrac{3}{4}})}\frac{2^{2\sfrac{p^\star}{r'}}}{K^{\sfrac{p^\star}{r'}}}+\|\psi\|_{L^{s}(\Omega_{K,\sfrac{3}{4}})}\frac{2^{2(p^\star-1-\sfrac{p^\star}{ s })}}{K^{p^\star-1-\sfrac{p^\star}{s}}}+\frac{2^{2p^\star}}{K^{p^\star}},\\
1+L^2+\frac{2^{\frac{5}{n-1}}}{K^{\frac{1}{n-1}}}
\big(1+\|\phi\|_{L^\infty(\Omega_{K,\sfrac{3}{4}})}\big) + \|\psi\|_{L^{s}(\Omega_{K,\sfrac{3}{4}})}
\frac{2^{2\left(\frac{s-n}{s(n-1)}\right)}}{K^{\frac{s-n}{s(n-1)}}} + \frac{2^{2n'}}{K^{n'}} \bigg\}.
\end{multline}
Now, we state the following:
\begin{claim}\label{claim J0 small}
One has $M_{k,\sfrac{3}{4}}\to0$ as $k\to\infty$. More precisely, there exists a positive constant $c=c(n,A)$ such that, for every $k>0$ satisfying
$$\int\limits_{\Omega_{k,\sfrac{3}{4}}}A(|\nabla u |)\dex\le (2\gamma_n)^{-n},$$
one has
$$M_{k,\sfrac{3}{4}}\le 2\int\limits_{\Omega_{k,\sfrac{3}{4}}}A(|\nabla  u  |)\dex+ c|\Omega_{k,\sfrac{3}{4}}|+A_n\big(2\big( (|u|-k)_+  \big){}_{\B_{\sfrac{3}{4}}}\big)|\Omega_{k,\sfrac{3}{4}}|.$$
\end{claim}
Recalling that
$$M_0=M_{\frac{K}{2},\frac{3}{4}}$$
and noting that, as a function of $K$, $c_K$ is bounded away from $0$, once the claim is established we can choose $K>0$ such that
\begin{equation*}
\widehat{c}L M_0^\frac{1}{n}\le1.
\end{equation*}
Then, the above argument yields the following recursive formula:
\begin{equation}\label{iterative formula}
M_{j+1}\le c_K \tau^{j} M_j^{1+\alpha}\quad\text{for all}\ j\in\N.
\end{equation}
It thus remains to prove that Claim \ref{claim J0 small} holds.
\begin{proof}[Proof of Claim \ref{claim J0 small}]
Since $u\in L^{B}_{\mathrm{loc}}(\Omega)$, $|\Omega_{k,\sfrac{3}{4}}|\to0$ as $k\to\infty$. Hence, by the dominated convergence theorem, the same happens to the integral
\begin{equation}\label{integral of gradient goes to zero}
	\int\limits_{\Omega_{k,\sfrac{3}{4}}}A(|\nabla  (|u|-k)_+  |)\dex.
\end{equation}
Then, we only have to prove that
$$\lim_{k\to\infty}\int\limits_{\Omega_{k,\sfrac{3}{4}}}A( (|u|-k)_+  )\dex=0.$$
One has:
\begin{multline}
	\int\limits_{\Omega_{k,\sfrac{3}{4}}}A(|u|-k)\dex\le c|\Omega_{k,\sfrac{3}{4}}|+\int\limits_{\Omega_{k,\sfrac{3}{4}}}A_n( (|u|-k)_+  )\dex\le\\
	\le c|\Omega_{k,\sfrac{3}{4}}|+A_n\big(2\big( (|u|-k)_+  \big){}_{\B_{\sfrac{3}{4}}}\big)|\Omega_{k,\sfrac{3}{4}}|+\\
	+\int\limits_{\Omega_{k,\sfrac{3}{4}}}A_n\big(2\big| (|u|-k)_+  -\big( (|u|-k)_+  \big){}_{\B_{\sfrac{3}{4}}}\big|\big)\dex\label{multline claim J0 small}
\end{multline}
where the first inequality comes from $A\lesssim A_n$ near infinity, and $c=c(n,A)$ is a positive constant. Let us show that the third integral on the right-hand side of \eqref{multline claim J0 small} tends to zero as $k\to\infty$. First, we note that, since the integral in \eqref{integral of gradient goes to zero} vanishes in the limit,
$$2\gamma_n\bigg(\int\limits_{\Omega_{k,\sfrac{3}{4}}}A(|\nabla u  |)\dex\bigg)^{\frac{1}{n}}\le 1$$
for $k$ large enough. Then, applying
Theorem \ref{Sobolev-Poincaré inequality} to $(|u|-k)_+$, we obtain
\begin{align*}
	\int\limits_{\Omega_{k,\sfrac{3}{4}}}A_n\big(2\big| (|u|-k)_+  -\big( (|u|-k)_+  \big){}_{\B_{\sfrac{3}{4}}}\big|\big)\dex&\le\int\limits_{\Omega_{k,\sfrac{3}{4}}}A_n\Bigg(\frac{\big| (|u|-k)_+  -\big( (|u|-k)_+  \big){}_{\B_{\sfrac{3}{4}}}\big|}{\gamma_n\big(\int_{\B_{\sfrac{3}{4}}}A(|\nabla  u  |)\dex\big)^{\frac{1}{n}}}\Bigg)\dex\\
	&\le\int\limits_{\Omega_{k,\sfrac{3}{4}}}A(|\nabla  u  |)\dex.
\end{align*}
Since the  last integral vanishes as $k\to\infty$, the conclusion follows.
\end{proof}

\subsection{Conclusion} Let $K>0$ be such that
$$\widehat{c}L M_0^\frac{1}{n}\le1,$$
so that \eqref{iterative formula} holds. The definitions of $(h_j)$ and $(\sigma_j)$ yield $h_j\le K$ and $\sigma_j\ge\frac{1}{2}$ for all $j\in\N$. Moreover
$$M_{K,\sfrac{1}{2}}\le \inf_{j\in\N}M_j=\lim_{j\to\infty}M_j.$$
If we can prove that $M_j\to0$ as $j\to\infty$ --- possibly after increasing $K$ if necessary --- we immediately deduce that $u$ is bounded in $\B_{\sfrac{1}{2}}$. Indeed, by the definition of $M_{K,\sfrac{1}{2}}$ and the previous formula, we obtain
$$\int\limits_{\B_{\sfrac{1}{2}}}A((|u|-K  )_+)\dex\le M_{K,\sfrac{1}{2}}=0,$$
which implies
\begin{equation}\label{u bounded from above}
|u|\le K\quad \text{\Ae in}\ \B_{\sfrac{1}{2}}.
\end{equation}
Therefore, to conclude it suffices to show that $(M_j)$ vanishes for sufficiently large $K$. We follow a classical idea going back to \cite{DeGiorgi57}: exploiting the recursive formula \eqref{iterative formula}, we prove that $(M_j)$ decays exponentially. More precisely, provided that
\begin{equation}
	c_K \tau^{\frac{1}{\alpha}}M_0^\alpha\le 1,\label{inequality to start the iteration}
\end{equation}
we claim
\begin{equation*}
M_j\le M_0 \tau^{-\frac{1}{\alpha}j}\quad\text{for all}\ j\in\N.
\end{equation*}
Setting $Y_j=\tau^{\frac{1}{\alpha}j} M_j$, it is sufficient to show that
\begin{equation}\label{iterative step}
Y_j\le Y_0\quad \text{for all}\ j\in\N,\ j\ge1.
\end{equation}
We proceed by induction on $j$. By \eqref{iterative formula},
$$Y_{j+1}\le c_K \tau^{\frac{1}{\alpha}} Y_j^{1+\alpha}\quad\text{for all}\ j\in\N.$$
For the base case, since $Y_0=M_0$, the inequality \eqref{inequality to start the iteration} entails
$$Y_1\le c_K \tau^{\frac{1}{\alpha}} Y_0^\alpha Y_0\le Y_0.$$
Assuming now that $Y_j\le Y_0$ for some $j\ge1$, we have
$$Y_{j+1}\le c_K \tau^{\frac{1}{\alpha}} Y_{j}^{1+\alpha}\le c_K \tau^{\frac{1}{\alpha}} Y_0^{1+\alpha}\le Y_0,$$
where the second inequality follows from the inductive hypothesis, and the third follows again from \eqref{inequality to start the iteration}. This establishes the claimed exponential decay of $(M_j)$, and therefore \eqref{u bounded from above} holds.
\section*{Acknowledgments}
The author is a member of GNAMPA of INdAM and is supported by the University of Florence, the University of Perugia, and INdAM. The author warmly thanks Andrea Cianchi for many valuable comments.
\printbibliography
\end{document}